\documentclass[12pt]{article}
\usepackage{a4wide,amsmath,amssymb,hyperref}
\usepackage{graphicx}
\usepackage{hyperref}
\usepackage{mathrsfs}
\usepackage{amsfonts,amsmath,amsthm, amssymb}
\usepackage{latexsym, euscript, epic, eepic, color}
\usepackage{indentfirst}
\usepackage{color}
\usepackage{enumerate}
\usepackage{cite}
\usepackage{footmisc}
\usepackage{verbatim}
\makeatletter
\@addtoreset{equation}{section}
\makeatother
\newtheorem{theorem}{Theorem}[section]
\newtheorem{lemma}{Lemma}[section]
\newtheorem{proposition}{Proposition}[section]

\newtheorem*{lthm}{Theorem LSV}

\newtheorem*{llw1thm}{Theorem LLW1}
\newtheorem*{llwconj}{Conjecture LLW}
\newtheorem*{llw2thm}{Theorem LLW2}
\newtheorem{corollary}{Corollary}[section]

\theoremstyle{definition}

\newtheorem{example}{Example}[section]

\makeatletter
\newcommand{\rmnum}[1]{\romannumeral #1}
\newcommand{\Rmnum}[1]{\expandafter\@slowromancap\romannumeral #1@}
\makeatother

\theoremstyle{remark}
\newtheorem{remark}{{Remark}}[section]

\renewcommand{\emph}[1]{{\it#1}}

\newcommand{\red}[1]{\textcolor[rgb]{1.00,0.00,0.00}{#1}}
\newcommand{\sv}[1]{{\color{blue}   {#1}}}

\DeclareMathOperator{\diam}{diam}

\begin{document}
\title{Intersecting well approximable and missing digit sets}

\author{Bing Li\textsuperscript{a}, Sanju Velani\textsuperscript{b} and Bo Wang\textsuperscript{c,d} \\
\small \it \textsuperscript{\rm a }School of Mathematics, South China University of Technology,\\[-1ex]
\small \it Guangzhou 510641, China\\
\small \it \textsuperscript{\rm b }Department of Mathematics, University of York,\\[-1ex]
\small \it Heslington, York, YO10 5DD, England\\
\small \it \textsuperscript{\rm c }School of Mathematics, Jiaying University,
\small \it Meizhou 514015, China\\
\small \it \textsuperscript{\rm d }School of Mathematics, Sun Yat-sen University,
\small \it Guangzhou 510275, China
}
\date{}
\maketitle
\begin{center}
\begin{minipage}{135mm}
{\footnotesize {\bf Abstract.}  Let $b\geq3$ be an integer and $C(b,D)$ be the set of real numbers in $[0,1]$ whose $b$-ary expansion consists of digits restricted  to a given set $D\subseteq\{0,\ldots,b-1\}$. Given an integer $t\geq2$ and a real, positive function $\psi$, let $W_{t}(\psi)$ denote the set of  $x$ in $[0,1]$ for which
 $|x-  p/t^{n}|<\psi(n)$ for infinitely many $(p,n)\in\mathbb{Z}\times\mathbb{N}$. We prove a general Hausdorff dimension result concerning  the intersection of $W_{t}(\psi)$ with an arbitrary self similar set which  implies that  $\dim_{\rm H}(W_{t}(\psi)\cap C(b,D))   \le    \dim_{\rm H}W_{t}(\psi)   \times \dim_{\rm H}C(b,D)$.   When $b$ and $t$ have the same prime divisors,  under certain restrictions on the digit set $D$,  we give a sufficient condition for the Hausdorff measure of $W_{t}(\psi)\cap C(b,D)$ to be zero.  This closes  a gap in a result of  Li, Li and Wu \cite{LLW2025} and shows that the dimension of the intersection can be strictly less than    the product of the dimensions.  The latter disproves the product conjecture of  Li, Li and Wu. 
}
\end{minipage}
\end{center}

\vskip0.5cm \begin{center} {\footnotesize {\bf Key words:} Diophantine approximation, missing digit sets, Hausdorff measure and dimension}
            \end{center}


\footnotetext{\small \it E-mails addresses: \rm scbingli@scut.edu.cn (B. Li), sanju.velani@york.ac.uk (S. Velani),
math\_bocomeon@163.com (B. Wang).}

\section{Introduction}
\subsection{Background and motivation}
In 1984, Mahler published an influential paper \cite{M1984} entitled `Some suggestions for further research', in which he raised the following problem:
\begin{itemize}
\item[``\!\!]How close can irrational elements of Cantor's set be approximated by rational numbers (i) in Cantor's set, and (ii)  by rational numbers not in Cantor's set?''
\end{itemize}
Mahler's problem has inspired a body of fundamental research and we refer the reader to \cite{ABCY2023,ACY2023,B2024, B2025,BHZ2024,BHZ2025,B2008,BD2016,FD2014,FD2015,
KL2023,LSV2007,LLW2025,S2021,TWW2024,WWX2017,W2001,Y2021} and references therein. In this paper, we do not attempt to give an exhaustive overview but instead concentrate on the key results that are relevant to the present work.  In short, the solution to  (i) in its simplest form  states:  as close as you like!  Indeed, this is a straightforward consequence of a natural Hausdorff measure criterion for well approximable sets  in which the denominators of the rational approximates are restricted to powers of three.    This will be our starting point.

Throughout,  let $C$   denote the standard  middle-third Cantor set; that is the set of real numbers in $[0,1]$ whose ternary expansion contains only the digits $0$ and $2$.  Also, for an integer $t \ge 2$ and
 a real, positive function $\psi:\mathbb{N}\to(0,\infty)$, we define
\begin{equation} \label{dioset}
W_{t}(\psi) :=\Big\{x\in[0,1]:\left|x-
p/t^n
\right|<\psi(n)\ {\rm for\ i.m.}\ (p,n)\in\mathbb{Z}\times\mathbb{N}\Big\}.
\end{equation}
where ``i.m.'' stands for ``infinitely many''. Thus,  $W_{t}(\psi)$  is the standard set of $\psi$-well approximable numbers in which the  denominators of the rational approximates $p/q$ are restricted to powers of $t$.  With $t=3$,  Levesley et al \cite[Theorem 1]{LSV2007} proved  the  following Hausdorff measure criterion for the size of the set $W_{3}(\psi)\cap C$.

\begin{lthm}
For any real number $s> 0 $,
\begin{equation}\label{19}
\begin{aligned}
\mathcal{H}^{s}(W_{3}(\psi)\cap C)=\begin{cases} 0  \ &{\rm if}\quad \sum_{n=1}^{\infty}\psi(n)^{s} \, 3^{\frac{\log 2}{\log 3}n}<\infty, \\[2ex]
\mathcal{H}^{s}(C) \ &{\rm if}\quad \sum_{n=1}^{\infty}\psi(n)^{s}  \, 3^{\frac{\log 2}{\log 3}n}=\infty. \end{cases}
\end{aligned}
\end{equation}
\end{lthm}

\noindent Throughout, for  a subset $X$ of $\mathbb{R}$,  we denote by  $\mathcal{H}^{s}(X)$ the $s$-dimensional Hausdorff measure and by  $ \dim_{\rm H} X$  the Hausdorff dimension of  $X$  -- see \S \ref{22} for the definitions.   Several comments are in order.
\begin{itemize}
  \item[$\bullet$]    The  theorem implies the following dimension statement:
\begin{equation}  \label{Cdimstate}
\dim_{\rm H}(W_{3}(\psi)\cap C)=  \frac{1}{\lambda_{\psi}} \times \frac{\log 2}{\log 3} =  \dim_{\rm H}W_{3}(\psi)   \times \dim_{\rm H}C ,
\end{equation}
where   $\lambda_\psi \ge 1 $ is given by  \eqref{dim-order} below with $t=3$.
  \item[$\bullet$]  If $|x-
p/3^n  |<\psi(n) $ for some $ x \in C $ and $\psi(n)\leq 3^{-n}  $, then it is easily verified that  the rational  $p/3^n$ must lie in $C$.  Thus, points in $W_{3}(\psi)\cap C$ are $\psi$-well approximable  with respect to ternary rational numbers in $C$.
   \item[$\bullet$] When $s = \log2/\log 3 = \dim_{\rm H}C$, the measure $\mathcal{H}^{s}$ is simply the standard Cantor measure $\mu$ supported on $C$.
\end{itemize}
%

\noindent Theorem~LSV can be viewed as the ``restricted'' analogue of the classical Jarn{\'\i}k-Besicovitch theorem - for this and a basic overview of  the classical theory of metric Diophantine approximation, see   \cite{MR3618787} and references therein.

The above theorem, naturally leads to the challenge of establishing an analogue for other denominators and missing digit sets. More specifically,  given (i) an integer $t\geq2$ and a function  $\psi:\mathbb{N}\to(0,\infty)$ we consider the general set given above by  \eqref{dioset} and (ii) given  an integer $b\geq3$ and a set $D \subseteq \{0,\ldots,b-1\}$ with cardinality   $\# D \ge 2$, we consider the missing digit set $C(b,D)$; that is,  the set of real numbers in $[0,1]$ whose $b$-ary expansion contains only the digits in $D$.  Before discussing the analogues of Theorem~LSV for $W_{t}(\psi)\cap C(b,D)$, we recall two well known facts:
\begin{equation}  \label{dim-cantorSV}
\dim_{\rm H} C(b,D) =  \gamma:=\frac{\log \#D}{\log b}
\end{equation}
and
\begin{equation}  \label{dim-order}
\dim_{\rm H}W_{t}(\psi) =   \min \Big\{  1, \frac{1}{\lambda_\psi}  \, \Big\},  \ \ \ \  {\rm where }  \ \   \lambda_\psi:= \liminf_{n \to \infty}  \frac{- \log \psi(n)}{n \log t }  \, .
\end{equation}
\noindent When investigating the intersection of $W_{t}(\psi)$ and $ C(b,D)$,  it is natural to consider three cases corresponding to whether or not $b$ and $t$ are multiplicatively independent (i.e., $\frac{\log b}{\log t}\notin\mathbb{Q}$) and if they are,  whether or not $b$ and $t$  have the same prime divisors.  Each case influences how the sets  $W_{t}(\psi)$ and $ C(b,D)$  behave and interact, based on the multiplicative properties of $b$ and $t$. This will become clear as the discussion unfolds.

Let us begin with the case that $b$ and $t$ are multiplicatively dependent.  Thus,  by definition $\frac{\log b}{\log t}\in\mathbb{Q}$ and for the moment let us further assume that $b=t$. In  \cite[\S7]{LSV2007}, the authors state  that the arguments used in the proof of their  Theorem~LSV ``can be modified in the obvious manner'' to yield the natural generalisation for arbitrary missing digit sets. Indeed, \cite[Theorem~4]{LSV2007} explicitly claims  that the analogue of \eqref{19} for $\mathcal{H}^{s}(W_{b}(\psi)\cap C(b,D))$  holds,  with the  sum  replaced by
 \begin{equation}  \label{sumlsvwrong}
\sum_{n=1}^{\infty}\psi(n)^{s}b^{\gamma n}  \, .  \end{equation}  However,  this turns out to be a careless oversight.
Li, Li and Wu \cite{LLW2025} demonstrated  that the general statement is not always valid.   Specifically, in   \cite[Example 3.1]{LLW2025}, they consider the case where   $b=5$, $D=\{1,2\}$ and $\psi(n)=\frac{1}{4}\cdot5^{-n}$.  They showed that
$$W_{5}(\psi)\cap C(5,\{1,2\})=\emptyset  \, $$
providing a concrete counterexample  to  \cite[Theorem~4]{LSV2007}.  Furthermore, they established  the following corrected form of the general statement which turns out to be  dependent on the quantity:
\begin{equation} \label{quantm}
m_{*}:=\min\left(\min D,b-1-\max D\right)  \, .
\end{equation}

\bigskip

\begin{llw1thm}
For any real number $s\geq0 $,
\begin{equation} \label{sumllw1}
\begin{aligned}
&\quad\mathcal{H}^{s}(W_{b}(\psi)\cap C(b,D)) \\[2ex] &
=  \begin{cases} 0  \ &{\rm if}\quad
\sum\limits_{n\geq1: \psi(n)>\frac{m_{*}}{(b-1)b^{n}}}\left(\psi(n) -\frac{m_{*}}{(b-1)b^{n}}\right)^{s}b^{\gamma n}<\infty, \\[4ex]
\mathcal{H}^{s}(C(b,D)) \ &{\rm if}\quad \sum\limits_{n\geq1: \psi(n) >\frac{m_{*}}{(b-1)b^{n}}}\left(\psi(n) -\frac{m_{*}}{(b-1)b^{n}}\right)^{s}b^{\gamma n}=\infty. \end{cases}
\end{aligned}
\end{equation}
\end{llw1thm}

\noindent In the above, if there are no terms in the sum (as is the case in the above counterexample)  we put the sum equal to zero.   Note that if $D$ contains at least one of the digits  $0$ and $b-1$, then by definition  $m_{*}=0$  and the sum appearing in \eqref{sumllw1} coincides with  \eqref{sumlsvwrong}.  In turn, Theorem~LLW1 implies that \cite[Theorem~4]{LSV2007} is correct with this extra assumption. Indeed, in this case much more is true.  Li, Li and Wu \cite[Theorem 1.5]{LLW2025}  show that the analogous statement is in fact  valid for any  multiplicatively dependent $b$ and $t$ -- not just when $b=t$.  
In other words, they proved that for real number $s\geq0$,
\begin{equation}  \label{llwbt}
\mathcal{H}^{s}(W_{t}(\psi)\cap C(b,D))=\begin{cases} 0  \ &{\rm if}\quad \sum\limits_{n=1}^{\infty}\psi(n)^{s}t^{\gamma n}<\infty, \\[2ex]
\mathcal{H}^{s}(C(b,D))   \ &{\rm if}\quad  \sum\limits_{n=1}^{\infty}\psi(n)^{s}t^{\gamma n}=\infty. \end{cases}
\end{equation}
Thus, it follows that if $b$ and $t$ are multiplicatively dependent and $D$ contains at least one of the digits $0$ and $b-1$, then \eqref{Cdimstate} is  true in general; that is
  \begin{equation}  \label{dimstate}
\dim_{\rm H}(W_{t}(\psi)\cap C(b,D))  \  =  \  \dim_{\rm H}W_{t}(\psi)   \times \dim_{\rm H}C(b,D)  \, .
\end{equation}

\bigskip

We now  turn our attention to  the case that $b$ and $t$ are multiplicatively independent.  Thus,  by definition $\frac{\log b}{\log t}\notin\mathbb{Q}$ and  in addition we assume that $b$ and $t$ have the same prime divisors.  In order to describe the current state of play, we let
\begin{equation}\label{18}
\begin{aligned}
\alpha_{1}= \alpha_{1}(b,t):=\min\left\{\frac{v_{q}(t)}{v_{q}(b)}:q\ {\rm is\ a\ prime\ divisor\ of}\ b\right\},  \\[3ex]
\alpha_{2}=\alpha_{2}(b,t):=\max\left\{\frac{v_{q}(t)}{v_{q}(b)}:q\ {\rm is\ a\ prime\ divisor\ of}\ b\right\},
\end{aligned}
\end{equation}
where $v_{q}(b)$ is the greatest integer such that $q^{v_{q}(b)}$ divides $b$.
With this notation in mind,  Li, Li and Wu \cite[Theorem 1.6]{LLW2025}  obtained a generalisation of Theorem~LLW1 which leads to the following cleaner statement under the assumption  that  $D$ contains at least one of the digits  $0$ and $b-1$.

\begin{llw2thm}
Suppose $b$ and $t$ have the same prime divisors  and the digit set  $D$ contains at least one of the digits $0$ and $b-1$.  Then, for any real number $s\geq0 $,
\begin{equation}  \label{eqn-thmllw}
\begin{aligned}
\mathcal{H}^{s}(W_{t}(\psi)\cap C(b,D))=\begin{cases} 0  \ &{\rm if}\quad  \sum\limits_{n=1}^{\infty} \psi(n)^{s} \ b^{\alpha_{2}\gamma n}<\infty, \\[3ex]
\mathcal{H}^{s}(C(b,D))  \ &{\rm if}\quad    \sum\limits_{n=1}^{\infty}\psi(n)^{s} \ b^{\alpha_{1}\gamma n}=\infty. \end{cases}
\end{aligned}
\end{equation}
\end{llw2thm}

%
%

Several comments are in order.    Firstly, we note that the  statement of Theorem~LLW2 includes the  case where  $b$ and $t$ are multiplicatively dependent. In this case, we have that  $ \alpha_1=  \alpha_2 =  \log t/\log b $,  and thus  sums in \eqref{eqn-thmllw} and  \eqref{llwbt} coincide.   Our interests, however, lies  in the case that $b$ and $t$ are multiplicatively independent.  In this setting,  it is readily verified (see the start o\S\ref{mainresSV} for the details) that
\begin{equation}\label{intersv}
  \alpha_1<\frac{\log t}{\log b}<\alpha_2,
\end{equation}
 which  of course implies the obvious  fact  that $ \alpha_1   <   \alpha_2 $.  The upshot is that the sums for divergence and convergence  in \eqref{eqn-thmllw} do not coincide and this creates a  genuine gap in  Theorem~LLW2. In this paper, we show under additional  restrictions on the digit set $D$ and the function $\psi$, that the $\mathcal{H}^{s}$-measure of $W_{t}(\psi)\cap C(b,D)$ is in fact determined by the behaviour of the  sum involving $\alpha_1$. In other words, under extra conditions,  we refine the  measure zero criterion in Theorem~LLW2  to its optimal form.  The precise statement is given by  Theorem~\ref{12} in \S\ref{mainresSV} below and it constitutes one of our main results.

 Regarding the dimension of  $W_{t}(\psi)\cap C(b,D)$,  if  $b$ and $t$ have the same prime divisors  and the digit set $D$ contains at least one of the digits $0$ and $b-1$,  then  Theorem~LLW2 implies   that
\begin{equation}\label{13}
\begin{aligned}
\frac{\alpha_{1}\log b}{\log t}\dim_{\rm H}W_{t}(\psi)\dim_{\rm H}C(b,D)&  \ \le \   \dim_{\rm H}(W_{t}(\psi)\cap C(b,D)) \\[1ex]
&~\hspace{-8ex}\le  \ \frac{\alpha_{2}\log b}{\log t}\dim_{\rm H}W_{t}(\psi)\dim_{\rm H}C(b,D).
\end{aligned}
\end{equation}
Note that when $b$ and $t$ are  multiplicatively dependent, \eqref{13} clearly coincides with the product formula \eqref{dimstate}  --  that is, the  dimension of the intersection is equal to the product of the dimensions.   On the other hand, note  that when $b$ and $t$ are multiplicatively independent, it  follows  via \eqref{intersv} that if  $\dim_{\rm H}W_{t}(\psi)>0$, then the  lower  bound  in  \eqref{13} for the  dimension of the intersection is strictly less than  the product of the dimensions, while the upper bound is strictly greater.
 Li, Li and Wu proposed in  \cite[Conjecture 6.3]{LLW2025} that the product formula still holds even when  $b$ and $t$ are multiplicatively independent.

\begin{llwconj} 
Suppose $b$ and $t$ have the same prime divisors  and the digit set  $D$ contains at least one of the digits $0$ and $b-1$.  Then,  the product formula  \eqref{dimstate} holds;  that is
$$
\dim_{\rm H}(W_{t}(\psi)\cap C(b,D))  \  =  \  \dim_{\rm H}W_{t}(\psi)   \times \dim_{\rm H}C(b,D)  \, .
$$
\end{llwconj}

\noindent We show that this conjecture is false. A straightforward consequence of our measure result (Theorem~\ref{12} in \S\ref{mainresSV} below) is that for any $b$ and $t$ have the same prime divisors,  there are digit sets $D$ for which the upper bound in \eqref{13}  can be improved so that it coincides with the lower bound;  that is
\begin{equation}   \label{dimintro}
\dim_{\rm H}(W_{t}(\psi)\cap C(b,D)) \  =  \ \frac{\alpha_{1}\log b}{\log t}\dim_{\rm H}W_{t}(\psi)\dim_{\rm H}C(b,D)  \, .
\end{equation}
In particular,  in view of \eqref{intersv},  this implies that the dimension of the intersection is strictly less than the product of the dimensions.  Indeed, it is not difficult to construct  functions $\psi $ for which the dimension of the intersection behaves as though the sets in question are “independent” or, equivalently “random”  -- see \cite[Chapter~8]{F2003}.  That is,
 \begin{equation} \label{randomdim} \dim_{\rm H}(W_{t}(\psi)\cap C(b,D)) \  =  \ \dim_{\rm H}W_{t}(\psi)+\dim_{\rm H}C(b,D)-1  \, .  \end{equation}
 For completeness, an explicit example of this  phenomenon  is provided  in  $\S\ref{mainresSV}\!\!: {\rm Remark}~\ref{randomdimrem}$.   Although obvious, it is worth noting that when  the dimension of the intersection  satisfies the product formula   \eqref{dimstate},  it is impossible for
$\dim_{\rm H}( W_{t}(\psi)\cap C(b,D))   $ to satisfy \eqref{randomdim} for any choice of  $\psi$ when both $W_{t}(\psi)$ and $C(b,D)$ have dimensions strictly less than one; that is, the interesting situation.     From a more general perspective ,  we establish a  result (see  \S\ref{mainresSV}: Corollary~\ref{110})  concerning the intersection of $W_{t}(\psi)$ with an  arbitrary self similar set, which implies  that
 \begin{equation}  \label{genupper}
\dim_{\rm H}(W_{t}(\psi)\cap C(b,D))  \ \le  \  \dim_{\rm H}W_{t}(\psi)   \times \dim_{\rm H}C(b,D)  \,
\end{equation}
  regardless of the values of $b$ and  $t$ and the composition of the digit set $D$.
  This clearly   improves the upper bound in \eqref{13}  irrespective of whether  or not $ b$ and $t$ have the same prime divisors, and weather or not $D$ contains either of the digits $0$ and $b-1$.

Although not addressed in this paper, for the sake of completeness,   we briefly comment on the  remaining case in which  $b$ and $t$ are multiplicatively independent and  do not share the same prime divisors.  Apart from the general upper bound \eqref{genupper}, which obviously holds in this case  
to the best of our knowledge,  existing measure theoretic  results are largely confined to the specific setting where     $b=3$, $t=2$, $D=\{0,2\}$ and   $\mathcal{H}^{s}= \mathcal{H}^{\gamma} $.  That is,  when $C(b,D)$ is the standard middle-third Cantor set $C$  and $\mathcal{H}^{s}$ is the standard Cantor measure $\mu$ supported on $C$.  In this context,   Allen et al  \cite[Theorem 2]{ABCY2023} have shown  that  $$ \mu(W_{2}(\psi_{\tau})\cap C)=0    \qquad {\rm if }  \qquad  \tau\geq\frac{0.922(1-\gamma)+1}{\gamma(2-\gamma)}  \,  ,  $$
where, for $\tau \ge 0$, the function $\psi_\tau$  is defined by  $ \psi_\tau(n) :=   2^{-n} n^{-\tau}  \, .
$
 On the other hand, Baker \cite{B2024} has established the complementary full-measure result:
  \begin{equation}  \label{bake}  \mu(W_{2}(\psi_{\tau})\cap C)= 1    \qquad {\rm if }  \qquad   0  \le \tau \le 0.01    \,  .  \end{equation}
 This result represents a significant refinement  over the first result of its type proved in \cite{ACY2023}.  In fact,    Baker \cite[Theorem 1.5]{B2024} proves  the following  much  stronger quantitative  version of the full measure statement:     for any sequence $(x_{n})_{n \in \mathbb{N}}$ of real numbers in $[0,1)$, we have
    $$\lim_{N\to\infty}\frac{\#\{1\le n\le N:\|2^{n}x-x_{n}\|<n^{-0.01}\}}{2\sum_{n=1}^{N}n^{-0.01}}=1   \, , $$
 for $\mu$-almost every $x\in C$, where $\| \ \cdot  \ \| $ denotes the distance to the nearest integer.  To see that this implies  \eqref{bake}, observe  that $x \in W_{2}(\psi_{\tau}) $   if and only if  $\#\{1\le n\le N:\|2^{n}x \|<n^{-\tau}\}   \to  \infty $ as $N \to \infty$.  In terms of dimension,  \eqref{bake} together with \eqref{dim-order}  implies that:   $$
    \dim_{\rm H}(W_{2}(\psi_{\tau} )\cap C)=\dim_{\rm H}C   =  \dim_{\rm H}C \times   \dim_{\rm H}(W_{2}(\psi_{\tau} )    \qquad {\rm if }  \qquad   0  \le \tau \le 0.01    \, .
    $$
    It turns out that Baker's approach can be  generalised and  refined   to   give the following broader statement.  Let  $b\geq3$  be a prime number and $t\geq2$ be an integer such that $b\nmid t$. 
    Let $\mu$ be any self-similar measure supported on   $C(b,D)$.   Then there exists an explicit, computable  constant $\tau_0>0$ such that for any  $\tau\in(0,\tau_0)$ and any sequence $(x_{n})_{n=1}^{\infty}$  of real numbers in $[0,1)$,
$$\lim_{N\to\infty}\frac{\#\{1\le n\le N:\|t^{n}x-x_{n}\|<n^{-\tau}\}}{2\sum_{n=1}^{N}n^{-\tau}}=1 $$
for $\mu$-almost  every $x\in C(b,D)$.   In the specific  setup considered by Baker (i.e.  $b=3$, $t=2$, $D=\{0,2\}$ and   $\mu= \mathcal{H}^{\gamma} $) the constant $ \tau_0 $  can be taken to be $0.013$  -- a modest improvement over Baker's bound of $0.01$.  We plan to provide the details of this and the general statement  in a forthcoming article.

\subsection{Main results}  \label{mainresSV}


In order to state our main measure theoretic result that closes the gap  in Theorem~LLW2,  we need to impose certain conditions on the digit set $D$.
With this in mind,  suppose $b$ and $t$ are multiplicatively independent and have the same prime divisors, say $q_{1},q_{2},\ldots,q_{K}$. Thus,
$$b=q_{1}^{v_{q_{1}}(b)}q_{2}^{v_{q_{2}}(b)}\ldots q_{K}^{v_{q_{K}}(b)}$$
and
$$t=q_{1}^{v_{q_{1}}(t)}q_{2}^{v_{q_{2}}(t)}\ldots q_{K}^{v_{q_{K}}(t)}$$
where as in  \eqref{18}   the quantity $v_{q}(b)$ is the greatest integer such that $q^{v_{q}(b)}$ divides $b$.    Before describing the required  conditions on $D$, which will  in part be determined by  the quantities $\alpha_1$ and $\alpha_2$ defined in  \eqref{18},  we first observe  that by definition we have that 
$$\alpha_1   \ \le \ \frac{v_{q_{j}}(t)}{v_{q_{j}}(b)}  \ \le  \ \alpha_2\qquad {\rm for\ all}\quad 1\le j\le K.$$  Hence, it follows that
$$ b^{\alpha_1}=  q_{1}^{\alpha_1v_{q_{1}}(b)}\times\ldots\times q_{K}^{\alpha_1v_{q_{K}}(b)}    \ \le  \ t  \ \le \  q_{1}^{\alpha_2v_{q_{1}}(b)}\times\ldots\times q_{K}^{\alpha_2v_{q_{K}}(b)}=b^{\alpha_2}.$$ Thus,
$\alpha_1\le\frac{\log t}{\log b}\le\alpha_2 $. Since    $\frac{\log t}{\log b}\notin\mathbb{Q}$ by the multiplicative independence of $b$ and $t$,    and  $\alpha_1,\alpha_2\in\mathbb{Q}$, we conclude that    \eqref{intersv} holds. Having verified \eqref{intersv}, we now return to the main  task  of describing the conditions on the digit set
$D$.    It follows from $\eqref{18}$ that there exists $1\le j\le K$, such that 
$$\frac{v_{q_{j}}(t)}{v_{q_{j}}(b)}=\alpha_{1}.$$
On the other hand, since  $b$ and $t$ are multiplicatively independent, there exists $1\le i\le K$, such that
\begin{equation*}
\frac{v_{q_{i}}(t)}{v_{q_{i}}(b)}>\alpha_{1}
\end{equation*}
Therefore, 
\begin{equation}\label{k_{*}}
1  \,  \le\,   k_{*}:=\#\left\{1\le i\le K: \frac{v_{q_{i}}(t)}{v_{q_{i}}(b)}>\alpha_{1}\right\}   \, \le  \, K-1  \, .  
\end{equation}
By reordering if necessary, without loss of generality,  we can assume that 
\begin{equation}\label{assumption}
 \frac{v_{q_{i}}(t)}{v_{q_{i}}(b)}>\alpha_{1},    \  \  \ \forall  \ \ 1\le i\le k_{*} \, , \qquad  {\rm and}\qquad \frac{v_{q_{j}}(t)}{v_{q_{j}}(b)}=\alpha_{1},   \ \ \ \forall\  \ k_{*}+1\le j\le K.
\end{equation}
With this is mind, let 
\begin{equation}\label{b_{*}}
b_{*}:=\frac{b}{q_{1}^{v_{q_{1}}(b)}\ldots q_{k_{*}}^{v_{q_{k_{*}}}(b)}}  \, , 
\end{equation}
and in turn, let 
\begin{equation}\label{D_{1}}
 D_{1}:=\left\{kb_{*}:\  1\leq k\leq\frac{b}{b_{*}}-1 \right\}
\end{equation}
and
\begin{equation}\label{D_{2}}
D_{2}:=\left\{kb_{*}-1:\  1\leq k\leq\frac{b}{b_{*}}-1 \right\}
\end{equation}
Note that $D_{2}= D_{1}-\{1\}$. We are now  in the position to state our main measure theoretic result. Throughout, $\lceil\cdot\rceil$ denotes as usual  the ceiling function.

\begin{theorem}\label{12sv}
Suppose $b$ and $t$  have the same prime divisors and  that $\psi:\mathbb{N}\to(0,\infty)$  satisfies $\psi(n)\le b^{-\lceil\alpha_{2}n\rceil-1}$  for $n$ sufficiently large.
\begin{itemize}
\item[(i)]  If the digit set  $D$ does not contain  $0$ and $b-1$, then
\begin{equation} \label{emptyre} W_{t}(\psi)\cap C(b,D)=\emptyset.
\end{equation}

    \item[(ii)]  If  $b$ and $t$ are multiplicatively independent and  the digit set    $D$  satisfies
\begin{equation}\label{restriction}
    D\subseteq\{0,1,\ldots,b-1\}\setminus(D_{1}\cup D_{2})   \, , 
\end{equation}
then,  for any real number $s\geq0$, 
\begin{equation*}
\mathcal{H}^{s}(W_{t}(\psi)\cap C(b,D))=  0 \quad  {\rm if}\quad  \sum\limits_{n=1}^{\infty} \psi(n)^{s}b^{\alpha_{1}\gamma n}<\infty  \, . 
\end{equation*}
In addition, if $\alpha_{1}\in\mathbb{N}$, then
$W_{t}(\psi)\cap C(b,D)=W_{b^{\alpha_{1}}}(\psi)\cap C(b,D).$

\end{itemize}
\end{theorem}
 
 \bigskip 

\begin{remark}
It is worth emphasizing that in part (ii) the quantities 
$b$ and $t$ are multiplicatively independent. Consequently, in the “In addition” statement $ t\neq b^{\alpha_{1}}$.  However, since $b$ and $b^{\alpha_{1}}$ are multiplicatively dependent  when $\alpha_1$ is an integer, the set equality together with  \eqref{llwbt} implies that the  complementary divergent statement holds;  that is 
$$
\mathcal{H}^{s}(W_{t}(\psi)\cap C(b,D))=  \infty  \quad  {\rm if}\quad  \textstyle{\sum_{n=1}^{\infty} } \psi(n)^{s}b^{\alpha_{1}\gamma n} = \infty  \, . 
$$
Although this is immediate, we also note that since
 $b_{*}\geq2$, it follows from the definition of $D_{1}$ and $D_{2}$ that both $0$ and $b-1$  belong to the right hand side of \eqref{restriction}. Consequently, we can always ensure that  $\#D\geq2$. 
\end{remark}

\bigskip 
 
 The first part  of Theorem~\ref{12sv} is essentially obvious but it does clarify the necessity of the  condition that  $D$   contains at least one of the digits  $0$ or $b-1$ in Theorem~LLW2.  It is worth pointing out that  Li, Li and Wu  gave a concrete example in the multiplicatively dependent case (\!\!\cite[Example~3.1]{LLW2025} in which  $b=t=5$ and $D=\{1,2\}$)  to show that the condition in their theorem is  necessary.  Part (i) shows that it is necessary generically even in  the  multiplicatively independent case. To the best of our knowledge,   it is not known if the condition  that  $D$  contains at least one of the digits  $0$ or $b-1$ is necessary in Theorem~LLW2 (or indeed Theorem~\ref{12} below) if $\psi:\mathbb{N}\to(0,\infty)$  satisfies     $\psi(n) > c \,   b^{-\lceil\alpha_{2}n\rceil-1}$ for $c >1 $ large. For the sake of completeness, we mention that, under the assumption that $D$ does not contain  $0$ and $b-1$,  part (i) can be slightly improved to the statement that  \eqref{emptyre} holds if for $n$ sufficiently large, 
 $$\psi(n)\le \frac{m_{*}}{(b-1)b^{\lceil\alpha_{2}n\rceil}}  \qquad  {\rm where }   \qquad m_{*}=\min\left(\min D,b-1-\max D\right)   \, . $$    To conclude this discussion concerning part (i) of the theorem, we note that that  the conclusion is not true if   $b$ and $t$ 
 and do not have the same prime divisors. To see this, let  $b $ be an odd number, $t $ be an even number and suppose that  the digit  $\frac{b-1}{2}\in D$. Then, it follows that the point 
$\frac{1}{2}\in C(b,D)  $
and for any $\psi$, we have that  
$$\textstyle{\left|\frac{1}{2}-\frac{\frac{1}{2}t^{n}}{t^{n}}\right|=0<\psi(n)}   \qquad {\rm for \  all \ }  n \in \mathbb{N}   \, .    $$
Therefore,  $\frac{1}{2}\in W_{t}(\psi)\cap C(b,D).$

The significance of condition \eqref{restriction} in the second part of Theorem~\ref{12sv}  will become evident in the course of the proof. Nonetheless, the specific example presented below in \S\ref{exposecond} serves to highlight both its presence  and perhaps our failure to notice the obvious.  The following ``gap'' free Hausdorff measure criterion for the size of $W_{t}(\psi)\cap C(b,D)$  is a direct consequence of combining Theorem~\ref{12sv}  with Theorem~LLW2.

\begin{theorem}\label{12}
Suppose $b$ and $t$ are multiplicatively independent and have the same prime divisors and that the  digit sets $D$ contains at least one of $0$ and $b-1$.   Furthermore, suppose that $D$ satisfies 
\eqref{restriction} and that $\psi:\mathbb{N}\to(0,\infty)$  satisfies $\psi(n)\le b^{-\lceil\alpha_{2}n\rceil-1}$  for $n$ sufficiently large.   Then,  for any real number $s\geq0$,
\begin{equation*}
\begin{aligned}
\mathcal{H}^{s}(W_{t}(\psi)\cap C(b,D))=\begin{cases}0 \ &{\rm if}\quad \sum\limits_{n=1}^{\infty} \psi(n)^{s}b^{\alpha_{1}\gamma n}<\infty, \\[2ex]
+ \infty \  &{\rm if}\quad \sum\limits_{n=1}^{\infty}\psi(n)^{s}b^{\alpha_{1}\gamma n}=\infty. \end{cases}
\end{aligned}
\end{equation*}
\end{theorem}


\bigskip

\noindent We emphasize that, in view of part (i) of Theorem~\ref{12sv}, it is absolutely necessary for $D$
 to contain at least one of the digits $0$ or 
$b-1$ for the divergent part of Theorem~\ref{12} to hold.
On the other hand, the growth condition on $\psi$, as well as the requirement that $D$ satisfies condition~\eqref{restriction}, are only needed for the convergent part of the theorem.

\begin{remark} 
Note that in view of the condition on $\psi$ in Theorem~\ref{12}, we  have that 
$$ \textstyle{
\sum\limits_{n=1}^{\infty}\psi(n)^{s}b^{\alpha_{1}\gamma n}   \ \le \   \sum\limits_{n=1}^{\infty} b^{(\alpha_{1}\gamma - \alpha_2 s)n }  \, . }
$$
Now  $ \alpha_1 < \alpha_2$  if   $b$ and $t$ are multiplicatively independent. Thus,   for the sum on the left to have any chance of diverging we must have that $ s\leq\frac{\alpha_1\gamma}{\alpha_2} < \gamma= \dim_{\rm H} C(b,D)$  and so by the  definition of Hausdorff dimension  $\mathcal{H}^{s}(C(b,D))=\infty$ in divergent case  of the theorem. 
\end{remark}

As a consequence of Theorem $\ref{12}$, we immediately obtain the following ``dimension'' statement which shows that Conjecture~LLW is false.

\begin{corollary}\label{16}
Suppose $b$ and $t$ are multiplicatively independent and have the same prime divisors and that the digit set $D$ contains at least one of $0$ and $b-1$.   Furthermore, suppose that $D$ satisfies 
\eqref{restriction} and that $\psi:\mathbb{N}\to(0,\infty)$  satisfies $\psi(n)\le b^{-\lceil\alpha_{2}n\rceil-1}$  for $n$ sufficiently large.   Then  
\begin{equation*}
\dim_{\rm H}(W_{t}(\psi)\cap C(b,D))=\frac{\alpha_{1}\log b}{\log t}\dim_{\rm H}W_{t}(\psi)\dim_{\rm H}C(b,D).
\end{equation*}
\end{corollary}

We now turn our attention to stating our main ``global'' lower bound dimension result  that deals with  the intersection of $W_{t}(\psi)$ with an  arbitrary non-empty subset $A$ of $[0,1]$. Given $A$ and $\delta > 0$,  let $N_{\delta}(A)$ be the smallest
number of sets of diameter at most $\delta$ which cover $A$. The following theorem give a sufficient condition for the $s$-dimensional Hausdorff measure of $W_{t}(\psi)\cap A$ to be zero.

\begin{theorem}\label{11}
Let $A$ be a non-empty subset of $[0,1]$ and   $\psi:\mathbb{N}\to(0,\infty)$ be a function. Then, for any real number $s>0$, 
\begin{equation*}
\mathcal{H}^{s}(W_{t}(\psi)\cap A)=0  \quad  {\rm if}  \quad  \sum\limits_{n=1}^{\infty}\psi(n)^{s}N_{t^{-n}}(A)<\infty  \, . 
\end{equation*}
\end{theorem}

\bigskip

\begin{remark}
 Observe that if $\psi(n)>t^{-n}/2$ for infinitely many $n\in\mathbb{N}$, then trivially $W_{t}(\psi)=[0,1]$. Consequently, for any $s > 0$, we have that  $\mathcal{H}^{s}(  W_{t}(\psi)\cap A) =  \mathcal{H}^{s}(A)   $.   However, the convergent sum condition imposes  a restriction on $s$, which in turn implies that $\mathcal{H}^{s}(A) = 0$.  
\end{remark}


The following  upper bound  statement for the Hausdorff dimension of $ W_{t}(\psi)\cap A  $ can be deduced from Theorem~\ref{11} and well know results concerning the packing dimension $\dim_{\rm P}X$ of a subset $X$ of $\mathbb{R}$ -- for details of the latter  see  \S\ref{22}. 

\begin{corollary}\label{17}
Let $A$ be a non-empty subset of $[0,1]$ and $\psi:\mathbb{N}\to(0,\infty)$ be a function. Then
\begin{equation*}
\dim_{\rm H}(W_{t}(\psi)\cap A)\le\dim_{\rm H}W_{t}(\psi)\dim_{\rm P}A.
\end{equation*}
\end{corollary}

A clear cut  result of  Falconer \cite[Theorem 4]{F1989}  states that if $A$ is a self-similar set,  then
$\dim_{\rm P}A=\dim_{\rm H}A $. 
Combining this  with Corollary~\ref{17} implies  that the Hausdorff dimension of the intersection of $ W_{t}(\psi)$ with an arbitrary self similar set is bounded above by the product of their individual dimensions.  We state this formally as a corollary.

\begin{corollary}\label{110}
Let $A$ be a self-similar set and $\psi:\mathbb{N}\to(0,\infty)$ be a function. Then
\begin{equation*}
\dim_{\rm H}(W_{t}(\psi)\cap A)\le\dim_{\rm H}W_{t}(\psi)\dim_{\rm H}A.
\end{equation*}
\end{corollary}

\medskip 

\noindent Missing digits sets $C(b,D)$ are well known examples of self-similar sets, see for instance \cite[Chapter 9]{F2003} and references therein.  Thus, the above corollary implies inequality  \eqref{genupper}  irrespective  of the values of $b$ and  $t$ and the composition of the digit set~$D$.

\subsubsection{An example: exposing condition \eqref{restriction}}   \label{exposecond}
We bring the section to  a close  with a concrete example that swifty pinpoints   our  reason  for condition \eqref{restriction} in the second part of Theorem~\ref{12sv}.  With this in mind, let $b=6$ and $t=12$.  Then by definition, it follows that $\alpha_1 = 1$ and  $\alpha_2 = 2$,  $D_1= \{3\}$   and  $D_2= \{2\}$,  and thus \eqref{restriction} holds   for any digit set  
$$D \subseteq \{0,1,4,5\}   \, .  $$
For the moment suppose $D \subseteq \{0,1,2,3,4,5\} $ is any set with four elements. Then, by  \eqref{dim-cantorSV} we have that $$ \textstyle{ \dim_{\rm H} C(6,D) =  \gamma=\frac{\log 4}{\log 6}  \, . }$$
Next observe that 
$$ \textstyle{
W_{12}(\psi) = \limsup\limits_{n\to \infty} A_n(\psi)   \quad {\rm where }   \quad  A_n(\psi):=
\bigcup\limits_{0\le p\le 12^{n}}  B\big(\frac{p}{12^{n}},\psi(n)\big)  \, \cap \, [0,1]} $$
Thus,  it follows  that determining an upper bound for the Hausdorff measure (or dimension)  of $W_{12}(\psi) \cap C(6,D)   $ boils down obtaining an upper bound for the cardinality of 
$$
\Gamma_n (\psi) := \left\{0\le p\le 12^{n}:  \textstyle{B\left(\frac{p}{12^{n}},\psi(n)\right)}\cap C(b,D)\neq\emptyset\right\}   \, . $$
Now, the condition that $\psi(n)\le b^{-\lceil\alpha_{2}n\rceil-1} = 6^{-2n-1}$  for $n$ sufficiently large,  allows us to conclude that
\begin{equation*}
\begin{aligned}
\Gamma_n (\psi)  & \ \subseteq \ \left\{0\le p\le 12^{n}:\frac{p}{12^{n}}=\frac{\xi_{1}}{6}+ \frac{\xi_{2}}{6^{2}}+\cdots+\frac{\xi_{2n}}{6^{2n}}\ \  {\rm with \ \  } \xi_{i} \in   D \right\} \\[2ex]
& ~\hspace{6ex} \cup \ \ \left\{0\le p\le 12^{n}:\frac{p}{12^{n}}=\frac{\xi_{1}}{6}+ \frac{\xi_{2}}{6^{2}}+\cdots+\frac{\xi_{2n}}{6^{2n}}+\frac{1}{6^{2n}}\ \  {\rm with \ \  } \xi_{i} \in   D \right\}.
\end{aligned}
\end{equation*}
In other words,  the balls centred at $12$-adic rationals  $p/12^n$ of radius $\psi(n)$ that intersect  $C(6,D)$ can only be those whose centres coincide with the endpoints of the basic intervals at the $2n$-th level of the Cantor-type construction  of $C(6,D)$.  Indeed, it turns out that if $D$ contains both $0$ and $b-1$ we actually have equality rather than just containment in the above.   In any case, the upshot is that   calculating   $\# \Gamma_n(\psi)$  boils down to calculating the cardinality of   
\begin{equation*}
\begin{aligned}
\Upsilon_n  & \ :=   \ \big\{(\xi_{n+1},\xi_{n+2},\ldots,\xi_{2n})  :3^{n}\mid6^{n-1}\xi_{n+1}+6^{n-2}\xi_{n+2}+\ldots+\xi_{2n} \ \  {\rm with \ \  } \xi_{i} \in   D    \big\} \\[2ex]
& ~\hspace{6ex} \cup \ \ \big\{ (\xi_{n+1},\xi_{n+2},\ldots,\xi_{2n}) :3^{n}\mid6^{n-1}\xi_{n+1}+6^{n-2}\xi_{n+2}+\ldots+\xi_{2n}+1 \ \  {\rm with \ \  } \xi_{i} \in   D    \big\}.
\end{aligned}
\end{equation*}
Since, with a little thought, it is not difficult to see that  $\# \Gamma_n (\psi) \,  \leq  \,  \# \Upsilon_n \, \times \, 4^n  $ .
In short, if $D$ satisfies condition \eqref{restriction}, so in this particular example   $D= \{0,1,4,5\}   $, we are able to show that 
$ \Upsilon_n =  \{(0,0,\ldots,0)\}   \cup  \{(5,5,\ldots,5)\}.$ Thus, 
$ \# \Gamma_n (\psi)=  2\times4^n   =      2\times6^{\gamma n}  $
and we are in great  shape -- indeed  it explains the presence of the ``$6^{\gamma n}$'' factor in the convergent sum condition in Theorem $\ref{12sv}$ that ensures that $\mathcal{H}^{s}(W_{12}(\psi)\cap C(6,D))=  0 $.    However, if $D$ does not satisfy \eqref{restriction} and  contains at least one of the digits $0$ and $b-1$ (so we are not in the trivial case covered by part (i) of Theorem~\ref{12sv}),   we are  unable to even show that  
$$\lim_{n\to\infty}\frac{\log\#  \Upsilon_n}{n}=0   \, ,$$
let alone that $\#  \Upsilon_n   \ll 1$.\footnote{{\small 
Throughout, given  functions $f$ and $g$  defined on a set $S$, we
write $f\ll g$ if there exists a constant $\kappa=\kappa(f,g,S)>0$, such that $|f(x)|\leq\kappa |g(x)|$ for all $x\in S$, and we write $f\asymp g$ if $f\ll g\ll f$.}} 
The former would suffice to yield the correct upper bound for the dimension of $W_{12}(\psi)\cap C(6,D)$.

\bigskip 

\begin{remark}  \label{randomdimrem}
With \eqref{randomdim} in mind, we use the above  example to show that for particular $\psi$ the dimension of the intersection $W_{12}(\psi)\cap C(6,D)$  behaves as though the sets in question are “independent” or, equivalently “random”.  By  Corollary $\ref{16}$, for  $D= \{0,1,4,5\}   $ and  any 
$\psi:\mathbb{N}\to(0,\infty)$ with $\psi(n)\le 6^{-2n-1}$  we have that 
\begin{equation*}
\begin{aligned}
&\quad\dim_{\rm H}(W_{12}(\psi)\cap C(6,D))=\frac{\log6}{\log12}\dim_{\rm H}W_{12}(\psi)\dim_{\rm H} C(6,D)=\frac{\log4}{\log12}\dim_{\rm H}W_{12}(\psi) .
\end{aligned}
\end{equation*}
Thus,  in view of \eqref{dim-order}, it follows that if
$$
\lambda_{\psi} = \frac{\log3\log6}{\log\frac{3}{2}\log12},
$$
then 
$$\dim_{\rm H}(W_{12}(\psi)\cap C(6,D))=\dim_{\rm H}W_{12}(\psi)+\dim_{\rm H}C(6,D)-1.$$

\noindent  To be absolutely explicit,  in the above we could  take $$ \psi: n \to  \psi(n):=6^{-\alpha n} \quad  {\rm  with}     \quad \alpha = \log3 /\log(3/2) = 2.7095\cdots\, . $$    We reiterate the fact  that when  $b$ and $t$ are multiplicatively dependent, and $D$ contains at least one of the digits $0$ and $b-1$,   the dimension of the intersection  satisfies the product formula   \eqref{dimstate}.  Thus,  it is impossible for
$\dim ( W_{t}(\psi)\cap C(b,D))   $ to satisfy \eqref{randomdim} for any choice of  $\psi$ when both $W_{t}(\psi)$ and $C(b,D)$ have dimensions strictly less than one; that is, the interesting situation.
\end{remark}

\bigskip

\section{Preliminaries: fractal measures and dimensions}\label{22}

In this section, for completeness and to establish notation, we briefly review standard concepts from fractal geometry that are used throughout the paper. We also describe an elegant result that relates the packing dimension to the box dimension. For further details, we refer the reader to \cite{F2003}. Throughout, given a non-empty bounded subset $U$ of $d$-dimensional Euclidean spake $\mathbb{R}^{d}$,  we let 
$\diam U$  denote the diameter of $U$ with respect to the Euclidean metric. Furthermore throughout, let $F$ be a subset of $\mathbb{R}^{d}$.

\begin{itemize}
    \item \emph{Hausdorff measure and dimension.\ }
For $\delta>0$, a countable (or finite) collection $\{U_{i}\}$ of sets in $\mathbb{R}^{d}$ of diameter at most $\delta$ that cover $F$ is called a $\delta$-cover of $F$. For  $s\geq0$, let
$$\mathcal{H}^{s}_{\delta}(F):=\inf\left\{\sum_{i}(\diam U_{i})^{s}:\{U_{i}\}\ {\rm is\ a}\  \delta {\rm -cover\ of}\ F\right\}.$$
The $s$-dimensional Hausdorff measure of $F$ is defined by 
\begin{equation*}
\mathcal{H}^{s}(F):=\lim_{\delta\to0}\mathcal{H}^{s}_{\delta}(F)  \, . 
\end{equation*}
In turn,  the  Hausdorff dimension of  $F$ is defined by 
\begin{equation*}
\dim_{\rm H}F:=\inf\{s:\mathcal{H}^{s}(F)=0\}=\sup\{s:\mathcal{H}^{s}(F)=\infty\}.
\end{equation*}
\item \emph{Box-counting dimensions. \ }   Suppose $F$ is bounded. 
For $\delta>0$, let $N_{\delta}(F)$ be the smallest number of sets of diameter at most $\delta$ which can cover $F$. The lower and upper box-counting dimensions of $F$ respectively are defined as
$$\underline{\dim}_{\rm B}F:=\liminf_{\delta\to0}\frac{\log N_{\delta}(F)}{-\log\delta}\quad {\rm and}\quad \overline{\dim}_{\rm B}F:=\limsup_{\delta\to0}\frac{\log N_{\delta}(F)}{-\log\delta}.$$
If these are equal we refer to the common value as the box-counting dimension or simply the  box dimension of $F$:
$$\dim_{\rm B}F:=\lim_{\delta\to0}\frac{\log N_{\delta}(F)}{-\log\delta}.$$

\item \emph{Packing measure and dimension.\ } 
Suppose $F$ is non-empty. Let $\delta>0$, a collection $\{B_{i}\}$ of disjoint balls of radius at most $\delta$ with centers in $F$ is called a $\delta$-packing of $F$. For $s\geq0$, let
$$\mathcal{P}_{\delta}^{s}(F):=\sup\left\{\sum_{i}(\diam B_{i})^{s}: \{B_{i}\}\ {\rm is\ a}\ \delta {\rm -packing}\  {\rm  of}\ F\right\}$$
and
$$\mathcal{P}_{0}^{s}(F):=\lim_{\delta\to0}\mathcal{P}_{\delta}^{s}(F).$$
The  \textit{$s$-dimensional packing measure} of $F$ is defined by 
$$\mathcal{P}^{s}(F):=\inf\left\{\sum_{i}\mathcal{P}_{0}^{s}(F_{i}):F\subseteq\bigcup_{i}F_{i}\right\},$$
where the infimum is taken over all possible countable covers $\{F_i\}$ of $F$. In turn,  the \textit{packing dimension} of $F$ is defined by
$$\dim_{\rm P}F:=\inf\{s:\mathcal{P}^{s}(F)=0\}=\sup\{s:\mathcal{P}^{s}(F)=\infty\}.$$
\end{itemize}

\noindent The following well known  proposition (see for instance  \cite[Proposition 3.8]{F2003}) brings to the forefront  between the upper box-counting dimension and the packing dimension.  As we shall soon see, it  plays a crucial role in the proof of Corollary~\ref{17}.
\begin{proposition}\label{23}
If $F$ is a non-empty subset of $\mathbb{R}^{d}$, then
$$\dim_{\rm P}F=\inf\left\{\sup_{i}\overline{\dim}_{\rm B}F_{i}:F\subseteq\bigcup_{i}F_{i}\right\},$$
where the  infimum   is taken over all possible countable covers $\{F_i\}$ of $F$.
\end{proposition}

\bigskip 

We now  present the proof of Corollary $\ref{17}$,  assuming the validity  of  Theorem~\ref{11}. 

\medskip 

\begin{proof}[\rm \textbf{Proof of Corollary~\ref{17} modulo Theorem~\ref{11}}]  In view of Proposition $\ref{23}$, the desired statement  follows on  showing that for every countable cover $\{U_{i}\}_{i=1}^{\infty}$ of $A$, we have that 
$$\dim_{\rm H}(W_{t}(\psi)\cap A)\leq\dim_{\rm H}W_{t}(\psi)\sup_{i\geq1}\overline{\dim}_{\rm B}U_{i}.$$
Since $W_{t}(\psi)\cap A\subseteq\bigcup_{i=1}^{\infty}(W_{t}(\psi)\cap U_{i})$, it follows from the monotonicity and countable stability properties of Hausdorff dimension that
\begin{equation}\label{36}
\dim_{\rm H}(W_{t}(\psi)\cap A)\leq\dim_{\rm H}\left(\bigcup_{i=1}^{\infty}(W_{t}(\psi)\cap U_{i})\right)=\sup_{i\geq1}\dim_{\rm H}(W_{t}(\psi)\cap U_{i}).
\end{equation}
In view of Theorem $\ref{11}$, for any $i\in\mathbb{N}$ and any real number $s>0$,  
\begin{equation}\label{equality6}
\mathcal{H}^{s}(W_{t}(\psi)\cap U_{i})=0  \quad  {\rm if}  \quad  \sum\limits_{n=1}^{\infty}\psi(n)^{s}N_{t^{-n}}(U_{i})<\infty  \,. 
\end{equation} 
Since
$$\psi(n)^{s}N_{t^{-n}}(U_{i})=t^{n\left(\frac{\log\psi(n)}{n\log t}s+\frac{\log N_{t^{-n}}(U_{i})}{n\log t}\right)},$$
it follows that
\begin{equation}\label{inequality1} 
\quad  \sum\limits_{n=1}^{\infty}\psi(n)^{s}N_{t^{-n}}(U_{i})<\infty  \quad  {\rm if}\quad s>  \frac{1}{\lambda_\psi}\limsup_{n\to\infty}\frac{\log N_{t^{-n}}(U_{i})}{n\log t},
\end{equation}
where
$$ \lambda_\psi= \liminf_{n \to \infty}  \frac{- \log \psi(n)}{n \log t }.$$
Thus, on combining \eqref{equality6}, \eqref{inequality1} and  \eqref{dim-order} we find that
$$\dim_{\rm H}(W_{t}(\psi)\cap U_{i}) \,  \leq  \,  \dim_{\rm H}W_{t}(\psi) \ \limsup_{n\to\infty}\frac{\log N_{t^{-n}}(U_{i})}{n\log t}  \, =  \,   \dim_{\rm H}W_{t}(\psi)   \ \overline{\dim}_{\rm B}U_{i}.$$
Therefore,
$$\sup_{i\geq1}\dim_{\rm H}(W_{t}(\psi)\cap U_{i})\leq\dim_{\rm H}W_{t}(\psi)\sup_{i\geq1}\overline{\dim}_{\rm B}U_{i}.$$
This together with $\eqref{36}$ gives
$$\dim_{\rm H}(W_{t}(\psi)\cap A)\leq\dim_{\rm H}W_{t}(\psi)\sup_{i\geq1}\overline{\dim}_{\rm B}U_{i}.$$
This thereby completes the proof of Corollary $\ref{17}$.
\textrm{}
\end{proof}

\section{Proofs of Theorems~\ref{12sv} and \ref{11}}\label{412}
Throughout this section, we suppose that the integers $b \ge 3$ and $t \ge 2$ have the same prime divisors.  Let $A$ be a non-empty subset of $[0,1]$. By the definition of $W_{t}(\psi)$, we have that
$$W_{t}(\psi) = \limsup\limits_{n\to\infty}\bigcup_{0\leq p\leq t^{n}}B\left(\frac{p}{t^{n}},\psi(n)\right)\cap[0,1]=\bigcap_{N=1}^{\infty}\bigcup_{n=N}^{\infty}\bigcup_{0\leq p\leq t^{n}} B\left(\frac{p}{t^{n}},\psi(n)\right)\cap[0,1] \, .$$
It thus follows  that   for any $N \ge 1 $,  the collection of balls 

$$\Big\{B\left(\frac{p}{t^{n}},\psi(n)\right) :n\geq N,\ p\in\Gamma_n (\psi,A)\Big\},$$
where 
$$
\Gamma_n (\psi,A) := \left\{0\le p\le t^{n}:  \textstyle{B\left(\frac{p}{t^{n}},\psi(n)\right)}\cap  A\neq\emptyset\right\}  \, , $$

\noindent is a natural covering of $W_{t}(\psi) \cap A$.   In turn, it follows that determining an upper bound for the Hausdorff measure (or dimension)  of $W_{t}(\psi) \cap A$  boils down obtaining an upper bound for the cardinality of 
$
\Gamma_n (\psi,A) $. 
For easy of notation, if $A=C(b,D)$, we will simply denote $\Gamma_{n}(\psi,A)$ by $\Gamma_{n}(\psi)$; that is 
$$
\Gamma_{n}(\psi)  :=  \Gamma_{n}(\psi,C(b,D))  .
$$
Also, we let
\begin{equation}\label{D_{*}}
D_{*}:=\{0,1,\ldots,b-1\}\setminus(D_{1}\cup D_{2}),
\end{equation}
where the sets $D_{1}$ and $D_{2}$ are defined in $\eqref{D_{1}}$ and $\eqref{D_{2}}$. 
The following proposition provides upper bounds for $\#\Gamma_n(\psi, A)$ and $\#\Gamma_n(\psi)$, which play key roles in the proofs of Theorem \ref{11} and Theorem \ref{12sv}, respectively.

\begin{proposition}\label{ub}
Let $b,t$ be integers with $b\geq3$ and $t\geq2$.
\begin{enumerate}
\item[(i)] If $\psi:\mathbb{N}\to(0,\infty)$ satisfies $\psi(n)\leq\frac{1}{2}t^{-n}$ for all $n\in\mathbb{N}$, then
$$\#\Gamma_n (\psi,A)  \le 2  \,   N_{t^{-n}}(A).    $$ 
\item[(ii)] Suppose $b$ and $t$ are multiplicatively independent and have the same prime divisors 
and that $D\subseteq D_{*}$.
If $\psi:\mathbb{N}\to(0,\infty)$ satisfies $\psi(n)\leq b^{-\lceil\alpha_{2}n\rceil-1}$ for all $n\in\mathbb{N}$, then
\begin{equation} \label{better}
    \#\Gamma_n (\psi)\ll b^{\alpha_{1}\gamma n}  \, 
\end{equation}
where
$$\gamma=\frac{\log\#D}{\log b}.$$
\end{enumerate}
\end{proposition}
\begin{remark}
It is easily verified that  
$$N_{t^{-n}}(C(b,D))\asymp t^{\gamma n}   \, . $$ Thus,   part (i) of Proposition~\ref{ub} implies that $\#\Gamma_n (\psi)\ll t^{\gamma n}$. On the other hand,     the upper bound  \eqref{better} in  part (ii) provides a sharper estimate,  since by definition $b^{\alpha_{1}}<t$.  
However, this improvement comes at the cost of imposing several additional assumptions.

\end{remark}


\subsection{Proof of Proposition \ref{ub}~(i)}
 By the definition of $N_{t^{-n}}(A)$, there exists a $t^{-n}$~-cover of $A$ with cardinality $N_{t^{-n}}(A)$. That is, there exists non-empty sets
$U_{1},\ldots,U_{N_{t^{-n}}(A)}$ with $\diam U_{i}\leq t^{-n}$ for each $1\leq i\leq N_{t^{-n}}(A)$, such that
$$A\subseteq\bigcup_{i=1}^{N_{t^{-n}}(A)}U_{i}.$$
Hence,
\begin{equation*}
\begin{aligned}
\Gamma_n(\psi,A) &\subseteq\left\{0\leq p\leq t^{n}: B\left(\frac{p}{t^{n}},\psi(n)\right)\cap\bigcup_{i=1}^{N_{t^{-n}}(A)}U_{i}\neq\emptyset\right\} \\[2ex]
&=\bigcup_{i=1}^{N_{t^{-n}}(A)}\left\{0\leq p\leq t^{n}: B\left(\frac{p}{t^{n}},\psi(n)\right)\cap U_{i}\neq\emptyset\right\} \\[2ex]
&=\bigcup_{i=1}^{N_{t^{-n}}(A)}\Gamma_n(\psi,U_{i}).
\end{aligned}
\end{equation*}
It follows that
\begin{equation}\label{34}
\#\Gamma_n(\psi,A) \leq\sum_{i=1}^{N_{t^{-n}}(A)}\#\Gamma_n(\psi,U_{i}).
\end{equation}
Now since $\psi(n)\leq\frac{1}{2}t^{-n}$, the balls $B\left(\frac{p}{t^{n}},\psi(n)\right)$  with $0\leq p\leq t^{n}$ are pairwise disjoint (the distance between their centers are at least $t^{-n}$) and so together with the fact that $\diam U_{i}\leq t^{-n}$, it follows that   $\#\Gamma_n(\psi,U_{i})\leq2$.   
Hence, we obtain via \eqref{34}   the desired statement:
$$\#\Gamma_n(\psi,A) \leq2N_{t^{-n}}(A). $$
\hfill $\Box$

\medskip 

As we shall see in 
\S\ref{pot}, the first part of the proposition (namely, the easy part) is essentially all that is required to establish Theorem \ref{11}.

\subsection{Proof of Proposition $\ref{ub}$ (ii)}

For establishing part (ii) of Proposition $\ref{ub}$, we will make use of  the following lemmas. The first lemma is pretty obvious  but nevertheless useful. 
\begin{lemma}\label{divisible 5}  
Suppose $b$ and $t$ have the same prime divisors.  Then,  
for any $n\in\mathbb{N}$,
$$t^{n}\mid b^{\lceil \alpha_{2}n\rceil}.$$
\end{lemma}
\begin{proof}[\rm \textbf{Proof}]
Recall that
$$b=q_{1}^{v_{q_{1}}(b)}q_{2}^{v_{q_{2}}(b)}\ldots q_{K}^{v_{q_{K}}(b)}  \quad 
{\rm and  }  \quad  t=q_{1}^{v_{q_{1}}(t)}q_{2}^{v_{q_{2}}(t)}\ldots q_{K}^{v_{q_{K}}(t)}.$$
Thus,
\begin{equation}\label{equality2}
b^{\lceil \alpha_{2}n\rceil}=q_{1}^{\lceil \alpha_{2}n\rceil v_{q_{1}}(b)}q_{2}^{\lceil \alpha_{2}n\rceil v_{q_{2}}(b)}\ldots q_{K}^{\lceil \alpha_{2}n\rceil v_{q_{K}}(b)}
\end{equation}
and
\begin{equation}\label{equality3}
t^{n}=q_{1}^{nv_{q_{1}}(t)}q_{2}^{nv_{q_{2}}(t)}\ldots q_{K}^{nv_{q_{K}}(t)}.
\end{equation}
Next note that for all $1\leq i\leq K$, it follows from the definition of $\alpha_{2}$ that $v_{p_{i}}(t)\leq\alpha_{2}v_{p_{i}}(b)$, and so
\begin{equation*}
\begin{aligned}
nv_{q_{i}}(t)\leq n\alpha_{2}v_{q_{i}}(b)\leq\lceil \alpha_{2}n\rceil v_{q_{i}}(b).
\end{aligned}
\end{equation*}
This together with \eqref{equality2} and \eqref{equality3} implies that  $
t^{n}$ divides $ b^{\lceil \alpha_{2}n\rceil}
  $ as desired. 
\textrm{}
\end{proof}

In order to the state the next lemma, we  need to introduce  some useful notation.  By definition, or equivalently construction,  we have that 
$$C(b,D)
=\bigcap_{m=1}^{\infty}C_{m}(b,D)   \qquad {\rm where }  \qquad C_{m}(b,D) :=   \bigcup_{i=1}^{(\#D)^{m}}I_{m,i} $$
is the union over the $m$-th level basic intervals
$$\Big\{I_{m,i}:= \Big[\textstyle{\frac{q_{i}}{b^{m}},\frac{q_{i}+1}{b^{m}}}\Big]
\ : \ i=1,2,\ldots,(\#D)^{m}\Big\}$$ 
associated with the  Cantor-type construction of $C(b, D)$.  With this in mind, let 
 $$E(m)=E(m,b,D): = \Big\{ \frac{q_{i}}{b^{m}},\frac{q_{i}+1}{b^{m}} 
\ : \ i=1,2,\ldots,(\#D)^{m}\Big\}$$ 
 denote the set of endpoints of the $m$-th level basic intervals  and in turn, for  $n,m\in\mathbb{N}$, let
$$F_{n,m}(\psi):=\left\{0\leq p\leq t^{n}: \exists\ \frac{q}{b^{m}}\in E(m)   \ \ {\rm s.t.}\ \ \left|\frac{p}{t^{n}}-\frac{q}{b^{m}}\right|<\max\Big(\psi(n),\frac{1}{b^{m}}\Big)\right\}.$$
Now, recall that  
$$
\Gamma_n (\psi) = \Gamma_{n}(\psi,C(b,D))= \left\{0\le p\le t^{n}:  \textstyle{B\left(\frac{p}{t^{n}},\psi(n)\right)}\cap  C(b,D)\neq\emptyset\right\}  \,  . $$
The following lemma plays an important role in that it enables us to  characterize $ \Gamma_n (\psi) $
in terms of the endpoints of the basic intervals  associated with $C(b, D)$. The latter are of course easier to describe.

\begin{lemma}\label{30} 
Let $b,t$ be integers with $b\geq3$ and $t\geq2$.
\begin{enumerate}
\item[(i)] Let $n,m\in\mathbb{N}$.  
Then  
$$\Gamma_{n}(\psi)\subseteq F_{n,m}(\psi).$$
\item[(ii)] Suppose $b$ and $t$ have the same prime divisors.  If $\psi:\mathbb{N}\to(0,\infty)$ satisfies $\psi(n)\leq b^{-\lceil\alpha_{2}n\rceil}$ for all $n\in\mathbb{N}$, then  
\begin{equation}\label{svweakform} \Gamma_{n}(\psi)\subseteq\left\{0\leq p\leq t^{n}:\frac{p}{t^{n}}\in E (\lceil\alpha_{2}n\rceil )\right\}.   \end{equation}
\end{enumerate}
\end{lemma}
\begin{proof}[\rm \textbf{Proof}] (i) Without loss of generality, we assume that $\Gamma_{n}(\psi)\neq\emptyset$. By definition, for every $p\in\Gamma_{n}(\psi)$, since $$C(b,D)\subseteq C_{m}(b,D)=\bigcup_{i=1}^{(\#D)^{m}}I_{m,i}  \, , $$ it follows that there exists $1\leq i\leq (\#D)^{m}$, such that 
$$ B\big(\textstyle{\frac{p}{t^{n}},\psi(n)}\big)\cap I_{m,i}\neq\emptyset.$$
We now consider two cases depending on  whether or not the endpoints $ \frac{q_{i}}{b^{m}},\frac{q_{i}+1}{b^{m}} $   of $I_{m,i}$ lie  in the ball $ B\left(\frac{p}{t^{n}},\psi(n)\right)$. 
\begin{itemize}
    \item {\em  There exists  one endpoint of $I_{m,i}$ that lies in 
$B\left(\frac{p}{t^{n}},\psi(n)\right)$.} Let us say $\frac{q_{i}}{b^{m}}\in B\left(\frac{p}{t^{n}},\psi(n)\right)$ without loss of generality. Then
$$\left|\frac{p}{t^{n}}-\frac{q_{i}}{b^{m}}\right|<\psi(n).$$
    \item {\em   Both endpoints of $I_{m,i}$ do not lie in $B\left(\frac{p}{t^{n}},\psi(n)\right)$.} In this case, we have 
$$B\big(\textstyle{\frac{p}{t^{n}},\psi(n)}\big) \subseteq I_{m,i}.$$
In particular, $\frac{p}{t^{n}}\in\left(\frac{q_{i}}{b^{m}},\frac{q_{i}+1}{b^{m}}\right)$. Thus,
$$\left|\frac{p}{t^{n}}-\frac{q_{i}}{b^{m}}\right|<\frac{1}{b^{m}}\quad {\rm and}\quad \left|\frac{p}{t^{n}}-\frac{q_{i}+1}{b^{m}}\right|<\frac{1}{b^{m}}.$$
\end{itemize}


\noindent On combining the above two cases, we obtain that there exists an endpoint $\frac{q}{b^{m}}\in E(m)$, such that $$\left|\frac{p}{t^{n}}-\frac{q}{b^{m}}\right|<\max\Big(\psi(n) ,\frac{1}{b^{m}}\Big).$$
Therefore, for every $p\in\Gamma_{n}(\psi)$ we have that  $p\in F_{n,m}(\psi)$ and so  this implies that 
$$\Gamma_{n}(\psi)\subseteq F_{n,m}(\psi).$$


\noindent  (ii) Fix $n\in\mathbb{N}$. Without loss of generality, we suppose that $\Gamma_{n}(\psi)\neq\emptyset$. It follows from Lemma $\ref{30}$ (i), on putting $m=\lceil \alpha_{2}n\rceil$,  that  
\begin{equation}\label{contain}
\Gamma_{n}(\psi)\subseteq F_{n,\lceil \alpha_{2}n\rceil}(\psi).
\end{equation}
 Let $p\in\Gamma_{n}(\psi)$. In view of \eqref{contain}, there exists $\frac{q}{b^{\lceil \alpha_{2}n\rceil}}\in E (\lceil\alpha_{2}n\rceil)$, such that
$$\left|\frac{p}{t^{n}}-\frac{q}{b^{\lceil \alpha_{2}n\rceil}}\right|<\max\Big(\psi(n),\frac{1}{b^{\lceil \alpha_{2}n\rceil}}\Big).$$
This together with the fact that $\psi(n)\leq b^{-\lceil \alpha_{2}n\rceil}$,  implies that
$$\left|\frac{p}{t^{n}}-\frac{q}{b^{\lceil \alpha_{2}n\rceil}}\right|<\frac{1}{b^{\lceil \alpha_{2}n\rceil}}.$$
Hence,
\begin{equation}\label{inequality}
\left|q-\frac{b^{\lceil \alpha_{2}n\rceil}}{t^{n}}p\right|<1.
\end{equation}
By Lemma \ref{divisible 5}, we know that $t^{n}\mid b^{\lceil \alpha_{2}n\rceil}$ and so the left-hand side of \eqref{inequality} is an integer. The upshot of this and \eqref{inequality} is that
$$q=\frac{b^{\lceil \alpha_{2}n\rceil}}{t^{n}}p.$$ Therefore,
$\ \displaystyle{\small{\frac{p}{t^{n}}=\frac{q}{b^{\lceil \alpha_{2}n\rceil}}  } \in E(\lceil\alpha_{2}n\rceil)}  \ $
as desired.
\textrm{}
\end{proof}

The next lemma brings into the play the quantities  $b_{*}$, $D_{1}$ and $D_{2}$  as  defined in \eqref{b_{*}}, \eqref{D_{1}} and \eqref{D_{2}}.  These are of course implicit in the statement of part (ii) of Proposition~\ref{ub}.
\begin{lemma}\label{32}
Suppose $b$ and $t$ are multiplicatively independent and have the same prime divisors  and let  $n\in\mathbb{N}$.

\begin{enumerate}
    \item[(i)] If $(\xi_{1},\xi_{2},\ldots,\xi_{n})\in(\{0,1,\ldots,b-1\}\setminus D_{1})^{n}$ and $b_{*}^{n}\mid\xi_{1}b^{n-1}+\xi_{2}b^{n-2} +\cdots+\xi_{n}$, then
\begin{equation*}
\xi_{1}=\xi_{2}=\cdots=\xi_{n}=0.
\end{equation*}
\item[(ii)]  If $(\xi_{1},\xi_{2},\ldots,\xi_{n})\in(\{0,1,\ldots,b-1\}\setminus D_{2})^{n}$ and 
$b_{*}^{n}\mid\xi_{1}b^{n-1}+\xi_{2}b^{n-2} +\cdots+\xi_{n}+1$, then
\begin{equation*}
\xi_{1}=\xi_{2}=\cdots=\xi_{n}=b-1.
\end{equation*}
\end{enumerate}
\end{lemma}

\bigskip

\begin{remark}
It is worth highlighting that by the definitions  of $D_{1} $ and $D_{2} $, we have 
\begin{equation}  \label{wedding}
    D_{2}=\{b-1\}-D_{1}   \, .
\end{equation}
In light of this observation, it will become apparent during the proof of the lemma that the two statements of the lemma are, in fact, equivalent. 
\end{remark}

\medskip 

\begin{proof}[\rm \textbf{Proof}]
To prove part (i) of the lemma,  we use  induction on $n\in\mathbb{N}$.   For $n=1$, we have that
$\xi_{1}\notin D_{1} $ and $b_{*}\mid\xi_{1}$. Thus,   by the  the definition  of $D_{1} $   (see \eqref{D_{1}}), we have that $\xi_{1}=0$ as desired. Now let $n\geq2$, and suppose that part~(i) is true for $n-1$. We now show  that it is true for $n$. Let $(\xi_{1},\xi_{2},\ldots,\xi_{n})\in(\{0,1,\ldots,b-1\}\setminus D_{1})^{n}$ with
$b_{*}^{n}\mid \xi_{1}b^{n-1}+\xi_{2}b^{n-2} +\cdots+\xi_{n-1}b+\xi_{n}$. In particular, $b_{*}\mid \xi_{1}b^{n-1}+\xi_{2}b^{n-2} +\cdots+\xi_{n-1}b+\xi_{n}$. Since $b_{*}\mid b$, it follows that
$$b_{*}\mid\xi_{1}b^{n-1}+\xi_{2}b^{n-2} +\cdots+\xi_{n-1}b$$
and so  $b_{*}\mid\xi_{n}$. This together with the fact that $\xi_{n}\notin D_{1}$ implies that   $\xi_{n}=0$. Hence, $b_{*}^{n}\mid \xi_{1}b^{n-1}+\xi_{2}b^{n-2} +\cdots+\xi_{n-1}b$, and so $$b_{*}^{n-1}\mid\frac{b}{b_{*}}( \xi_{1}b^{n-2}+\xi_{2}b^{n-3}+\cdots+\xi_{n-1}).$$ Now, $\gcd\left(b_{*}^{n-1},\frac{b}{b_{*}}\right)=1$
thus we must have that 
$$b_{*}^{n-1}\mid \xi_{1}b^{n-2}+\xi_{2}b^{n-3}+\cdots+\xi_{n-1}.$$
In view of the inductive hypothesis, we deduce that  $$\xi_{1}=\xi_{2}=\cdots=\xi_{n-1}=0.$$ This completes the induction step and so establishes  part (i) of the lemma.   \\

 We now turn our attention to proving part (ii) of the lemma.  
Let $(\xi_{1},\xi_{2},\ldots,\xi_{n})\in(\{0,1,\ldots,b-1\}\setminus D_{2})^{n}$ with
$b_{*}^{n}\mid\xi_{1}b^{n-1}+\xi_{2}b^{n-2} +\cdots+\xi_{n} +1$. 
Then, using the relationship \eqref{wedding} linking $D_{1} $ and $D_{2} $, it follows that
$$\big( (b-1)-\xi_{1},(b-1)-\xi_{2},\ldots,(b-1)-\xi_{n}  \big )\in\big(\{0,1,\ldots,b-1\}\setminus D_{1}\big)^{n}.$$
Furthermore, since  $b_{*}\mid b$, it follows that
\begin{equation*}
\begin{aligned}
b_{*}^{n}\mid & b^{n}-(\xi_{1}b^{n-1}+\xi_{2}b^{n-2} +\cdots+\xi_{n}+1) \\[2ex]
&=(b-1)b^{n-1}+(b-1)b^{n-2} +\cdots+b-1+1-(\xi_{1}b^{n-1}+\xi_{2}b^{n-2} +\cdots+\xi_{n}+1)\\[2ex]
&=(b-1-\xi_{1})b^{n-1}+(b-1-\xi_{2})b^{n-2}+\cdots+b-1-\xi_{n}.
\end{aligned}
\end{equation*}
In view of part (i) of Lemma \ref{32}, we have  that 
$$b-1-\xi_{1}=b-1-\xi_{2}=\cdots=b-1-\xi_{n}=0,$$
and so 
$\xi_{1}=\xi_{2}=\cdots=\xi_{n}=b-1$ as desired.
\textrm{}
\end{proof}

We now pause before stating and proving our  final lemma  (from which  part (ii) of Proposition $\ref{ub}$ will follow easily), in order  to consider the following claim.
\begin{itemize}
    \item[]\textbf{Claim:}   {\em if $\alpha_1$ {\rm(}defined in \eqref{18}{\rm)} and $n \in \mathbb{N}$  are  such that  $\alpha_1n $ is an integer,  then  under the assumption that  $ D\subseteq D_{*}$,   Lemma~\ref{32}  enables us to  replace  $\alpha_2 $ in \eqref{svweakform}  by $\alpha_1$.}
\end{itemize}
 This strengthen of Lemma~\ref{30} (ii)  is significant, since it implies that if $ D\subseteq D_{*}$,   then   
$$\#\Gamma_{n}(\psi)   \ \le \  \#E(\alpha_1n)   \ \ll \    (\#D)^{\alpha_{1}n}=b^{\gamma\alpha_{1}n}   \, ,    
$$
where $ \gamma $ as usual is given by   \eqref{dim-cantorSV}.  
In other words,  we obtain  the desired upper bound  \eqref{better}  appearing in part (ii) of Proposition $\ref{ub}$. Indeed, this totally  completes the proof of the proposition in the case where  $\alpha_1 $ is an integer.  In this sense,  the claim provides the motivation for  the ``all powerful''  final lemma.  

To establish the claim,  we note that under the  condition that $\psi(n)\leq  b^{-\lceil\alpha_{2}n\rceil}$,  Lemma \ref{30} (ii) implies that   \eqref{svweakform} holds; that is to say that for any $p\in\Gamma_{n}(\psi)$, there exists $(\xi_{1},\ldots,\xi_{\lceil\alpha_{2}n\rceil})\in D^{\lceil\alpha_{2}n\rceil}$, such that
$$\frac{p}{t^{n}}=\frac{\xi_{1}}{b}+\frac{\xi_{2}}{b^{2}}+\cdots+\frac{\xi_{\lceil\alpha_{2}n\rceil}}{b^{\lceil\alpha_{2}n\rceil}}=\frac{\xi_{1}b^{\lceil\alpha_{2}n\rceil-1}+\xi_{2}b^{\lceil\alpha_{2}n\rceil-2}+\cdots+\xi_{\lceil\alpha_{2}n\rceil}}{b^{\lceil\alpha_{2}n\rceil}}$$
or
$$\frac{p}{t^{n}}=\frac{\xi_{1}}{b}+\frac{\xi_{2}}{b^{2}}+\cdots+\frac{\xi_{\lceil\alpha_{2}n\rceil}}{b^{\lceil\alpha_{2}n\rceil}}+\frac{1}{b^{\lceil\alpha_{2}n\rceil}}=\frac{\xi_{1}b^{\lceil\alpha_{2}n\rceil-1}+\xi_{2}b^{\lceil\alpha_{2}n\rceil-2}+\cdots+\xi_{\lceil\alpha_{2}n\rceil}+1}{b^{\lceil\alpha_{2}n\rceil}}$$
Furthermore, by Lemma \ref{divisible 5}, we know that $t^{n}\mid b^{\lceil\alpha_{2}n\rceil}$. 
Thus, is follows that 
$$\frac{b^{\lceil\alpha_{2}n\rceil}}{t^{n}} \mid \xi_{1}b^{\lceil\alpha_{2}n\rceil-1}+\xi_{2}b^{\lceil\alpha_{2}n\rceil-2}+\cdots+\xi_{\lceil\alpha_{2}n\rceil}  $$
or
$$\frac{b^{\lceil\alpha_{2}n\rceil}}{t^{n}}\mid \xi_{1}b^{\lceil\alpha_{2}n\rceil-1}+\xi_{2}b^{\lceil\alpha_{2}n\rceil-2}+\cdots+\xi_{\lceil\alpha_{2}n\rceil}+1.$$
Now, from the definition of $\alpha_{1}$ \eqref{18} and $k_{*}$ \eqref{k_{*}}, we have  that
\begin{equation}\label{equality}
t=b^{\alpha_{1}}\prod_{i=1}^{k_{*}}q_{i}^{v_{q_{i}}(t)-\alpha_{1}v_{q_{i}}(b)}.
\end{equation}
Recall that $v_{q}(b)$ is the greatest integer such that $q^{v_{q}(b)}\mid b$. It follows from \eqref{equality} and the definition of $b_{*}$ (see \eqref{b_{*}}) that
\begin{equation}\label{equality1}
\frac{b^{\lceil\alpha_{2}n\rceil}}{t^{n}}=b_{*}^{\lceil\alpha_{2}n\rceil-\alpha_{1}n}\prod_{i=1}^{k_{*}}q_{i}^{v_{q_{i}}(b)\lceil\alpha_{2}n\rceil-v_{q_{i}}(t)n}.
\end{equation}
If $\alpha_{1}n\in\mathbb{N}$, then 
$\lceil\alpha_{2}n\rceil-\alpha_{1}n\in\mathbb{N}$.  This together with \eqref{equality1} allows us to conclude that
\begin{equation}\label{divisible 6}  
b_{*}^{\lceil\alpha_{2}n\rceil-\alpha_{1}n}\mid\frac{b^{\lceil\alpha_{2}n\rceil}}{t^{n}}.
\end{equation}
Then under the condition that $D\subseteq D_{*}$, we can apply Lemma \ref{32} to conclude that
$$\xi_{\alpha_{1}n+1}=\xi_{\alpha_{1}n+2}=\cdots=\xi_{\lceil\alpha_{2}n\rceil}=0\quad {\rm if}\quad \frac{b^{\lceil\alpha_{2}n\rceil}}{t^{n}}\mid  \xi_{1}b^{\lceil\alpha_{2}n\rceil-1}+\xi_{2}b^{\lceil\alpha_{2}n\rceil-2}+\cdots+\xi_{\lceil\alpha_{2}n\rceil}$$
and  
$$\xi_{\alpha_{1}n+1}=\xi_{\alpha_{1}n+2}=\cdots=\xi_{\lceil\alpha_{2}n\rceil}=b-1\quad {\rm if}\quad \frac{b^{\lceil\alpha_{2}n\rceil}}{t^{n}}\mid \xi_{1}b^{\lceil\alpha_{2}n\rceil-1}+\xi_{2}b^{\lceil\alpha_{2}n\rceil-2}+\cdots+\xi_{\lceil\alpha_{2}n\rceil}+1.$$
The upshot is  that $\frac{p}{t^n}  \in E(\alpha_1 n)$  and we have shown that if $\alpha_{1}n\in\mathbb{N}$ then 
\begin{equation}  \label{svstrongform} \Gamma_{n}(\psi)\subseteq\left\{0\leq p\leq t^{n}:\textstyle{\frac{p}{t^{n}}}\in E (\alpha_{1}n)\right\}  \end{equation}
as claimed.   Note that in the  above argument, when   establishing the claim, 
it is absolutely paramount that   $\alpha_{1}n\in\mathbb{N}$. Indeed, we cannot  apply   Lemma~\ref{32}    if   $\alpha_{1}n\notin\mathbb{N}$. To overcome this deficiency 
we need to work a little  harder  from a technical point of view.  The following lemma  provides the appropriate general analogue of the claim and in turn a suitable strengthening of Lemma~\ref{32} under the assumption that  $ D\subseteq D_{*}$.  First a little more notation.   Recall, that by definition, in general $\alpha_{1}\in\mathbb{Q}$.  With this in mind, we write  
\begin{equation}\label{equality 4}
\alpha_{1}=\frac{l_{1}}{l_{0}}  \quad  {\rm where}\quad  l_{0},l_{1}\in\mathbb{N}\quad  {\rm and}\quad  \gcd(l_{0},l_{1})=1.  
\end{equation}
In turn, given  $ n \in  \mathbb{N}$, we  can then write 
\begin{equation}\label{equality 5} 
n=l_{0}\tilde{n}+r \quad {\rm where}\quad \tilde{n}:=\left\lfloor\frac{n}{l_{0}}\right\rfloor\quad {\rm and}\quad r:=n-l_{0}\left\lfloor\frac{n}{l_{0}}\right\rfloor.
\end{equation}

\medskip

\begin{lemma}\label{important}
Suppose $b$ and $t$ are multiplicatively independent and have the same prime divisors. If  $\psi:\mathbb{N}\to(0,\infty)$  satisfies $\psi(n)\le b^{-\lceil\alpha_{2}n\rceil-1}$  for all $n\in\mathbb{N}$, then
\begin{equation}\label{qwrt}\Gamma_{n}(\psi)\subseteq G_{n}:=\left\{0\le p\le t^{n}:\frac{p}{t^{n}}\in E (\lceil\alpha_{2}l_{0}\tilde{n}\rceil+\lceil\alpha_{2}r\rceil) \right\}.
\end{equation}
Furthermore,  if in addition  $D\subseteq D_{*}$, then
\begin{equation}\label{contain1}
\begin{aligned}
    G_{n} 
   &\subseteq\left\{0\le p\le t^{n}:\frac{p}{t^{n}}\in E(\alpha_{1}l_{0}\tilde{n} +\lceil\alpha_{2}r\rceil)   \right\}.
\end{aligned}
\end{equation}
\end{lemma}

\bigskip

\begin{remark}\label{trivial}
The conclusion \eqref{contain1} is trivial for the case $n<l_{0}$. Indeed, in this case, we have $\tilde{n}=\left\lfloor\frac{n}{l_{0}}\right\rfloor=0$ and so  by definition,  
$G_{n}$ coincides with the right hand side of  \eqref{contain1}.  

\end{remark}

 \bigskip 
With Lemma~\ref{important} at hand it 
is easy to establish  Proposition \ref{ub} (ii). 
Indeed, it follows from Lemma~\ref{important} that
$$\#\Gamma_{n}(\psi)\leq\#G_{n}\leq(\#D)^{\alpha_{1}l_{0}\tilde{n}+\lceil\alpha_{2}r\rceil}\leq(\#D)^{\alpha_{1}l_{0}\tilde{n}+\alpha_{2}r+1}\leq(\#D)^{\alpha_{1}l_{0}\tilde{n}+\alpha_{2}(l_{0}-1)+1}.$$
The last inequality follows from the fact that  $r\leq l_{0}-1$. Now  $l_{0}\tilde{n}\leq n$, so the upshot is that 
$$\#\Gamma_{n}(\psi)\ll b^{\alpha_{1}\gamma n}   $$
where $ \gamma $ as usual is given by   \eqref{dim-cantorSV}.   This is precisely the  upper bound  \eqref{better}  appearing in part (ii) of Proposition $\ref{ub}$ and so we are done.

 \bigskip 
 
\begin{proof}[\rm \textbf{Proof of Lemma \ref{important}}]
Fix $n\in\mathbb{N}$. Let $m_{0}:=\lceil\alpha_{2}l_{0}\tilde{n}\rceil+\lceil\alpha_{2}r\rceil$. Recall that $l_{0}$, $\tilde{n}$ and $r$ are defined in \eqref{equality 4} and 
\eqref{equality 5}. It follows from Lemma $\ref{30}$ with $m=m_{0}$, that
\begin{equation*} 
\Gamma_{n}(\psi)\subseteq F_{n,m_{0}}(\psi) \, . 
\end{equation*}
Thus,   \eqref{qwrt}  follows on showing that 
\begin{equation}\label{contain2}
 F_{n,m_{0}}(\psi)\subseteq G_{n}. 
\end{equation}
Without loss of generality, we suppose that $ F_{n,m_{0}}(\psi)\neq\emptyset$. Otherwise, \eqref{contain2} is trivial. For any $p\in F_{n,m_{0}}(\psi)$, there exists $\frac{q}{b^{m_{0}}}\in E(m_{0})$, such that
$$\left|\frac{p}{t^{n}}-\frac{q}{b^{m_{0}}}\right|<\max\Big(\psi(n),\frac{1}{b^{m_{0}}}\Big).$$
Since
\begin{equation*}
\psi(n)\leq b^{-(1+\lceil\alpha_{2} n\rceil)} =b^{-(1+\lceil\alpha_{2}(l_{0}\tilde{n}+r)\rceil)}\leq b^{-(\lceil\alpha_{2}l_{0}\tilde{n}\rceil+\lceil\alpha_{2}r\rceil)}=b^{-m_{0}},
\end{equation*}
where the last inequality makes use of the  fact that $\lceil x\rceil+\lceil y\rceil\leq\lceil x+y\rceil+1$ for all $x,y\in\mathbb{R}$, it follows that
\begin{equation}\label{37}
\left|\frac{p}{t^{n}}-\frac{q}{b^{m_{0}}}\right|<\frac{1}{b^{m_{0}}}.
\end{equation}
Then we show that $\frac{b^{m_{0}}}{t^{ n}}\in\mathbb{N}$. By Lemma \ref{divisible 5}, we have  that
\begin{equation}\label{qaz} t^{n}\mid b^{\lceil \alpha_{2}n\rceil}  \, .
\end{equation}
Furthermore, by the fact that $\lceil x+y\rceil\leq\lceil x\rceil+\lceil y\rceil$ for all $x,y\in\mathbb{R}$, we have that 
$$\lceil\alpha_{2}n\rceil=\lceil\alpha_{2}l_{0}\tilde{n}+\alpha_{2}r\rceil   \ \leq  \ \lceil\alpha_{2}l_{0}\tilde{n}\rceil+\lceil\alpha_{2}r\rceil=m_{0}.$$
It follows that $b^{\lceil \alpha_{2}n\rceil}\mid b^{m_{0}}$. This together with  
\eqref{qaz}
implies that
\begin{equation}\label{divisible}
t^{n}\mid b^{m_{0}}.
\end{equation}
The upshot of \eqref{37} and \eqref{divisible}  is that
$$\frac{p}{t^{n}}=\frac{q}{b^{m_{0}}}\in E(m_{0})  \,  $$
and this validates \eqref{contain2} as desired.

\medskip 

For the ``furthermore'' part of the lemma, by Remark \ref{trivial}  we can assume  that $n\geq l_{0}$ and without loss of generality that $G_{n}\neq\emptyset$. Recall that $m_{0}=\lceil\alpha_{2}l_{0}\tilde{n}\rceil+\lceil\alpha_{2}r\rceil$. For any $p\in G_{n}$, we have $\frac{p}{t^{n}}\in E(m_{0})$. That is, $\frac{p}{t^{n}}$ is the endpoint of some $m_{0}$-th level  basic intervals for $C(b,D)$.  The goal is to show that  
\begin{equation}   \label{svqaz}
\frac{p}{t^{n}}\in E(\alpha_{1}l_{0}\tilde{n} +\lceil\alpha_{2}r\rceil)  \, .    
\end{equation}
For this we consider two cases. 
\medskip 

\noindent \textbf{Case 1:} The point $\frac{p}{t^{n}}$ is the  left endpoint of some $m_{0}$-th basic intervals for $C(b,D)$. In this case,
there exists $(\xi_{1},\ldots,\xi_{m_{0}})\in D^{m_{0}}$, such that
$$\frac{p}{t^{n}}=\frac{\xi_{1}}{b}+\frac{\xi_{2}}{b^{2}}+\cdots+\frac{\xi_{ m_{0}}}{b^{m_{0}}}=\frac{\xi_{1}b^{m_{0}-1}+\xi_{2}b^{m_{0}-2}+\cdots+ \xi_{ m_{0}}}{b^{m_{0} }}.$$
By \eqref{divisible}, we know that $\frac{b^{m_{0}}}{t^{ n}}\in\mathbb{N}$. Thus,
\begin{equation}\label{divisible1}
\frac{b^{m_{0}}}{t^{n}}\mid\xi_{1}b^{m_{0}-1}+\xi_{2}b^{m_{0}-2}+\cdots+ \xi_{ m_{0}} .
\end{equation}
We  now show that 
\begin{equation}\label{divisible2}
b_{*}^{\lceil\alpha_{2}l_{0}\tilde{n}\rceil-\alpha_{1}l_{0}\tilde{n}}\mid\frac{b^{m_{0}}}{t^{n}}.
\end{equation}
It follows from \eqref{equality} that
\begin{equation*}
\begin{aligned}
\frac{b^{\lceil\alpha_{2}l_{0}\tilde{n}\rceil}}{t^{l_{0}\tilde{n}}}&=\frac{b^{\lceil\alpha_{2}l_{0}\tilde{n}\rceil}}{b^{\alpha_{1}l_{0}\tilde{n}}\prod_{i=1}^{k_{*}}q_{i}^{(v_{q_{i}}(t)-\alpha_{1}v_{q_{i}}(b))l_{0}\tilde{n}}} \\[2ex]
&=\frac{b^{\lceil\alpha_{2}l_{0}\tilde{n}\rceil-\alpha_{1}l_{0}\tilde{n}}}{\prod_{i=1}^{k_{*}}q_{i}^{v_{q_{i}}(b)(\lceil\alpha_{2}l_{0}\tilde{n}\rceil-\alpha_{1}l_{0}\tilde{n})}}\times
\frac{\prod_{i=1}^{k_{*}}q_{i}^{v_{q_{i}}(b)(\lceil\alpha_{2}l_{0}\tilde{n}\rceil-\alpha_{1}l_{0}\tilde{n})}}{\prod_{i=1}^{k_{*}}q_{i}^{(v_{q_{i}}(t)-\alpha_{1}v_{q_{i}}(b))l_{0}\tilde{n}}}  \\[2ex]
&=b_{*}^{\lceil\alpha_{2}l_{0}\tilde{n}\rceil-\alpha_{1}l_{0}\tilde{n}}\prod_{i=1}^{k_{*}}q_{i}^{v_{q_{i}}(b)\lceil\alpha_{2}l_{0}\tilde{n}\rceil-v_{q_{i}}(t)l_{0}\tilde{n}}.
\end{aligned}
\end{equation*}
Since $\frac{v_{q_{i}}(t)}{v_{q_{i}}(b)}\leq \alpha_{2}$ for each $1\leq i\leq k_{*}$, we have that 
$$v_{q_{i}}(b)\lceil\alpha_{2}l_{0}\tilde{n}\rceil-v_{q_{i}}(t)l_{0}\tilde{n}\geq v_{q_{i}}(b)\alpha_{2}l_{0}\tilde{n}-v_{q_{i}}(t)l_{0}\tilde{n}=(v_{q_{i}}(b)\alpha_{2}-v_{q_{i}}(t))l_{0}\tilde{n}\geq 0,$$
which implies that 
$$\prod_{i=1}^{k_{*}}q_{i}^{v_{q_{i}}(b)\lceil\alpha_{2}l_{0}\tilde{n}\rceil-v_{q_{i}}(t)l_{0}\tilde{n}}\in\mathbb{N}.$$
Hence,
$$b_{*}^{\lceil\alpha_{2}l_{0}\tilde{n}\rceil-\alpha_{1}l_{0}\tilde{n}}\mid\frac{b^{\lceil\alpha_{2}l_{0}\tilde{n}\rceil}}{t^{l_{0}\tilde{n}}}.$$
This together with the fact that $t^{r}\mid b^{\lceil\alpha_{2}r\rceil}$ (see Lemma~\ref{divisible 5}), $n=l_{0}\tilde{n}+r$ and $m_{0}=\lceil\alpha_{2}l_{0}\tilde{n}\rceil+\lceil\alpha_{2}r\rceil$ implies  \eqref{divisible2}. The upshot of \eqref{divisible1} and \eqref{divisible2} is that
\begin{equation}\label{divisible3}
\begin{aligned}
b_{*}^{\lceil\alpha_{2}l_{0}\tilde{n}\rceil-\alpha_{1}l_{0}\tilde{n}}&\mid\xi_{1}b^{m_{0}-1}+\xi_{2}b^{m_{0}-2} +\cdots+\xi_{\alpha_{1}l_{0}\tilde{n}+\lceil\alpha_{2} r\rceil}b^{\lceil\alpha_{2}l_{0}\tilde{n}\rceil-\alpha_{1}l_{0}\tilde{n}} \\[2ex]
& \ \ \ +   \ \xi_{\alpha_{1}l_{0}\tilde{n}+\lceil\alpha_{2} r\rceil+1}b^{\lceil\alpha_{2}l_{0}\tilde{n}\rceil-\alpha_{1}l_{0}\tilde{n}-1}+  \xi_{\alpha_{1}l_{0}\tilde{n}+\lceil\alpha_{2} r\rceil+2}b^{\lceil\alpha_{2}l_{0}\tilde{n}\rceil-\alpha_{1}l_{0}\tilde{n}-2}+\cdots+\xi_{m_{0}}.
\end{aligned}
\end{equation}
Since $b_{*}\mid b$, it follows that
$$b_{*}^{\lceil\alpha_{2}l_{0}\tilde{n}\rceil-\alpha_{1}l_{0}\tilde{n}}\mid\xi_{1}b^{m_{0}-1}+\xi_{2}b^{m_{0}-2} +\cdots+\xi_{\alpha_{1}l_{0}\tilde{n}+\lceil\alpha_{2} r\rceil}b^{\lceil\alpha_{2}l_{0}\tilde{n}\rceil-\alpha_{1}l_{0}\tilde{n}} .$$
This together with \eqref{divisible3} gives 
\begin{equation}  \label{10Jan}
b_{*}^{\lceil\alpha_{2}l_{0}\tilde{n}\rceil-\alpha_{1}l_{0}\tilde{n}}\mid\xi_{\alpha_{1}l_{0}\tilde{n}+\lceil\alpha_{2} r\rceil+1}b^{\lceil\alpha_{2}l_{0}\tilde{n}\rceil-\alpha_{1}l_{0}\tilde{n}-1}+  \xi_{\alpha_{1}l_{0}\tilde{n}+\lceil\alpha_{2} r\rceil+2}b^{\lceil\alpha_{2}l_{0}\tilde{n}\rceil-\alpha_{1}l_{0}\tilde{n}-2}+\cdots+\xi_{m_{0}} .
\end{equation}
Now, since  $n\geq l_{0}$ it follows that $\tilde{n}=\left\lfloor\frac{n}{l_{0}}\right\rfloor\geq1$, which implies that
$$\lceil\alpha_{2}l_{0}\tilde{n}\rceil-\alpha_{1}l_{0}\tilde{n}\geq\alpha_{2}l_{0}\tilde{n}-\alpha_{1}l_{0}\tilde{n}=(\alpha_{2}-\alpha_{1})l_{0}\tilde{n}>0.$$
Thus, $\lceil\alpha_{2}l_{0}\tilde{n}\rceil-\alpha_{1}l_{0}\tilde{n}\in\mathbb{N}$. Since $$ \xi_{\alpha_{1}l_{0}\tilde{n}+\lceil\alpha_{2} r\rceil+1},\xi_{\alpha_{1}l_{0}\tilde{n}+\lceil\alpha_{2} r\rceil+2},\ldots,\xi_{m_{0}}\in D\subseteq D_{*},$$ we have that $$ \xi_{\alpha_{1}l_{0}\tilde{n}+\lceil\alpha_{2} r\rceil+1},\xi_{\alpha_{1}l_{0}\tilde{n}+\lceil\alpha_{2} r\rceil+2},\ldots,\xi_{m_{0}}\notin D_{1}.$$ As a consequence of Lemma \ref{32} (i), it follows  that
$$\xi_{\alpha_{1}l_{0}\tilde{n}+\lceil\alpha_{2} r\rceil+1}=\xi_{\alpha_{1}l_{0}\tilde{n}+\lceil\alpha_{2} r\rceil+2}=\cdots=\xi_{m_{0}}=0   \, , $$
and thus
\begin{equation*}
\begin{aligned}
\frac{p}{t^{n}}&=\frac{\xi_{1}}{b}+\frac{\xi_{2}}{b^{2}}+\cdots+\frac{\xi_{ m_{0}}}{b^{m_{0}}}\\[2ex] 
&=\frac{\xi_{1}}{b}+\frac{\xi_{2}}{b^{2}}+\cdots+\frac{\xi_{\alpha_{1}l_{0}\tilde{n}+\lceil\alpha_{2} r\rceil}}{b^{\alpha_{1}l_{0}\tilde{n}+\lceil\alpha_{2} r\rceil}} \in \  E(\alpha_{1}l_{0}\tilde{n} +\lceil\alpha_{2}r\rceil).
\end{aligned}
\end{equation*}
In other words,  \eqref{svqaz} is true when we are in \textbf{Case 1}.

\medskip 

\noindent \textbf{Case 2:} The point $\frac{p}{t^{n}}$ is the right endpoint of some $m_{0}$-th basic intervals for $C(b,D)$. In this case,
there exists $(\xi_{1},\ldots,\xi_{ m_{0}})\in D^{m_{0}}$, such that
$$\frac{p}{t^{n}}=\frac{\xi_{1}}{b}+\frac{\xi_{2}}{b^{2}}+\cdots+\frac{\xi_{ m_{0}}}{b^{m_{0}}}+\frac{1}{b^{m_{0}}}=\frac{\xi_{1}b^{m_{0}-1}+\xi_{2}b^{m_{0}-2}+\cdots+ \xi_{ m_{0}}+1}{b^{m_{0}}}.$$
It follows from \eqref{divisible} that
\begin{equation*} 
\frac{b^{m_{0}}}{t^{n}}\mid\xi_{1}b^{m_{0}-1}+\xi_{2}b^{m_{0}-2}+\cdots+ \xi_{ m_{0}}+1.
\end{equation*} 
This is the Case 2 analogue of  \eqref{divisible1}  and on modifying the argument leading to \eqref{10Jan} in the obvious manner, we obtain that
$$b_{*}^{\lceil\alpha_{2}l_{0}\tilde{n}\rceil-\alpha_{1}l_{0}\tilde{n}}\mid\xi_{\alpha_{1}l_{0}\tilde{n}+\lceil\alpha_{2} r\rceil+1}b^{\lceil\alpha_{2}l_{0}\tilde{n}\rceil-\alpha_{1}l_{0}\tilde{n}-1}+  \xi_{\alpha_{1}l_{0}\tilde{n}+\lceil\alpha_{2} r\rceil+2}b^{\lceil\alpha_{2}l_{0}\tilde{n}\rceil-\alpha_{1}l_{0}\tilde{n}-2}+\cdots+\xi_{m_{0}}+1.$$
Then, since $$ \xi_{\alpha_{1}l_{0}\tilde{n}+\lceil\alpha_{2} r\rceil+1},\xi_{\alpha_{1}l_{0}\tilde{n}+\lceil\alpha_{2} r\rceil+2},\ldots,\xi_{m_{0}}\in D\subseteq D_{*},$$ we have that $$ \xi_{\alpha_{1}l_{0}\tilde{n}+\lceil\alpha_{2} r\rceil+1},\xi_{\alpha_{1}l_{0}\tilde{n}+\lceil\alpha_{2} r\rceil+2},\ldots,\xi_{m_{0}}\notin D_{2}.$$
As a consequence of Lemma \ref{32} (ii), it follows that 
$$ \xi_{\alpha_{1}l_{0}\tilde{n}+\lceil\alpha_{2} r\rceil+1}=\xi_{\alpha_{1}l_{0}\tilde{n}+\lceil\alpha_{2} r\rceil+2}=\cdots=\xi_{m_{0}}=b-1 , $$
and thus 
\begin{equation*}
\begin{aligned}
\frac{p}{t^{n}}&=\frac{\xi_{1}}{b}+\frac{\xi_{2}}{b^{2}}+\cdots+\frac{\xi_{ m_{0}}}{b^{m_{0}}}+\frac{1}{b^{m_{0}}} \\[2ex]
&=\frac{\xi_{1}}{b}+\frac{\xi_{2}}{b^{2}}+\cdots+\frac{\xi_{\alpha_{1}l_{0}\tilde{n}+\lceil\alpha_{2} r\rceil}}{b^{\alpha_{1}l_{0}\tilde{n}+\lceil\alpha_{2} r\rceil}}+\frac{b-1}{b^{\alpha_{1}l_{0}\tilde{n}+\lceil\alpha_{2} r\rceil+1}}+\frac{b-1}{b^{\alpha_{1}l_{0}\tilde{n}+\lceil\alpha_{2} r\rceil+2}}+\cdots+\frac{b-1}{b^{m_{0}}}+\frac{1}{b^{m_{0}}} 
\\[2ex]
&=\frac{\xi_{1}}{b}+\frac{\xi_{2}}{b^{2}}+\cdots+\frac{\xi_{\alpha_{1}l_{0}\tilde{n}+\lceil\alpha_{2} r\rceil}}{b^{\alpha_{1}l_{0}\tilde{n}+\lceil\alpha_{2} r\rceil}}+\frac{b^{\lceil\alpha_{2}l_{0}\tilde{n}\rceil-\alpha_{1}l_{0}\tilde{n}}-1}{b^{m_{0}}}+\frac{1}{b^{m_{0}}}
\\[2ex]
&=\frac{\xi_{1}}{b}+\frac{\xi_{2}}{b^{2}}+\cdots+\frac{\xi_{\alpha_{1}l_{0}\tilde{n}+\lceil\alpha_{2} r\rceil}}{b^{\alpha_{1}l_{0}\tilde{n}+\lceil\alpha_{2} r\rceil}}+\frac{1}{b^{\alpha_{1}l_{0}\tilde{n}+\lceil\alpha_{2} r\rceil}}\in \  E(\alpha_{1}l_{0}\tilde{n} +\lceil\alpha_{2}r\rceil).
\end{aligned}
\end{equation*}
In other words,  \eqref{svqaz} is true when we are in \textbf{Case 2}. This completes the proof of the lemma. 
\textrm{}
\end{proof}

As already shown before giving the above proof,  with Lemma~\ref{important} at hand it  is easy to establish  Proposition \ref{ub} (ii). 


\subsection{The finale: proving the theorems}  \label{pot}

Having completed the necessary groundwork in the previous section, we are now ready to prove Theorems~\ref{12sv} and \ref{11}.

\begin{proof}[\rm \textbf{Proof of Theorem~\ref{12sv}}]  Naturally, we deal with parts (i) and (ii) separately. 
\medskip

\noindent \textbf{Part~(i).}  We are given that $\psi(n)\leq b^{-\lceil \alpha_{2}n\rceil-1}$ for all $n$ sufficiently large. However,   we can assume,without loss of generality, that this inequality holds for all $n\in\mathbb{N}$.  Indeed, if this was not the case, we may  instead consider the modified function $\psi^*$ defined  by  $\psi^{*}(n):=\min(\psi(n),b^{-\lceil \alpha_{2}n\rceil-1})$.  Then,  by construction  $\psi^{*}(n)=\psi(n)$ for  all $n$ sufficiently large, and so 
$W_{t}(\psi)  = W_{t}(\psi^*)$.   Thus, it suffices to  establish \eqref{emptyre} with $\psi$ replaced by  $\psi^{*}$.  
Now, with this assumption and the definition of $\Gamma_{n}(\psi)$ in mind, in order to prove part (i), it suffices to show that   
$$\Gamma_{n}(\psi)=\emptyset    \qquad {\rm for \ all \ }  \, n\in\mathbb{N}.$$
Assume, for contradiction, that there exists $n\in\mathbb{N}$ such that
$$\Gamma_{n}(\psi)\neq\emptyset.$$
That is, there exists $0\leq p \leq t^{n}$, such that
\begin{equation}\label{neq}
B\left(\frac{p}{t^{n}},\psi(n)\right)\cap C(b,D)\neq\emptyset.   
\end{equation}
Then, by  Lemma $\ref{30}$ (ii)  it follows that  $\frac{p}{t^{n}}\in E(\lceil\alpha_{2}n\rceil)$.   Moreover, since   $D$ does not contain $0$ and $b-1$, we obtain  
$$B\left(\frac{p}{t^{n}},b^{-\lceil\alpha_{2}n\rceil-1}\right)\cap C_{\lceil \alpha_{2}n\rceil+1}(b,D)=\emptyset \, ,$$
where, recall, by definition   $C_m (b,D)$  is  the union over the $m$-th level basic intervals associated with the Cantor-type construction of $C (b,D)$.
Therefore, 
$$B\left(\frac{p}{t^{n}},b^{-\lceil\alpha_{2}n\rceil-1}\right)\cap C (b,D)=\emptyset.$$
Since $\psi(n)\leq b^{-\lceil \alpha_{2}n\rceil-1}$, it follows that
$$B\left(\frac{p}{t^{n}}, \psi(n)\right)\cap C (b,D)=\emptyset,$$
which contradicts \eqref{neq}. This completes the proof of part (i).

\bigskip

\noindent \textbf{Part~(ii).} For any $\delta>0$, since $\psi(n)\leq b^{-\lceil \alpha_{2}n\rceil-1}$ for all $n$ sufficiently  large, there exists $N_{0}\in\mathbb{N}$, such that for all $N\geq N_{0}$,
$$\left\{
 B\left(\frac{p}{t^{n}},\psi(n)\right)\cap C(b,D): \ n\geq N,p\in\Gamma_{n}(\psi)\right\}$$
is a $\delta$-cover of $W_{t}(\psi)\cap C(b,D)$. It follows that for any $s\geq0$ and any $N\geq N_{0}$, we have 
\begin{equation*}
\mathcal{H}_{\delta}^{s}(W_{t}(\psi)\cap C(b,D))  \, \leq   \, \sum_{n=N}^{\infty}\sum_{p\in\Gamma_{n}(\psi)} \psi(n) ^{s}  \, =   \, \sum_{n=N}^{\infty} \psi(n) ^{s}   \ \#\Gamma_{n}(\psi).
\end{equation*}
This together with Proposition \ref{ub} (ii) implies  that
\begin{equation}\label{35}
\mathcal{H}_{\delta}^{s}(W_{t}(\psi)\cap C(b,D))\ll\sum_{n=N}^{\infty}\psi(n)^{s}b^{n\alpha_{1}\gamma}
\end{equation}
for all $s\geq0$ and all $N\geq N_{0}$.  The upshot is  that if  $\sum_{n=1}^{\infty}\psi(n)^{s}b^{\alpha_{1}\gamma n} $ converges,  then the sum in \eqref{35} tends to zero as $N \to \infty$, and so by definition
$$\mathcal{H}^{s}(W_{t}(\psi)\cap C(b,D))=0   \, , $$
as desired.    It now remains to establish the `In addition'' statement  within part (ii). 
With this in mind, let $\alpha_1 \in\mathbb{N}$. Then, for any $n \in\mathbb{N}$, by the definition of $G_{n}$ (see \eqref{qwrt}), it is easily verified that 
\begin{equation*}
\begin{aligned}
G_{n}
&=\left\{0\le p\le t^{n}:\frac{p}{t^{n}}\in E(\lceil\alpha_{2} n\rceil) \right\}  \, . 
\end{aligned}
\end{equation*}
Thus, Lemma \ref{important} implies that 
$$\Gamma_{n}(\psi)\subseteq G_{n}$$
and 
\begin{equation*}
\begin{aligned}
G_{n}&\subseteq\left\{0\le p\le t^{n}:\frac{p}{t^{n}}\in E(\alpha_{1} n)\right\}. 
\end{aligned}
\end{equation*}
Hence,
\begin{equation*}
\begin{aligned}
 \bigcup_{p\in\Gamma_{n}(\psi)}\left(B\left(\frac{p}{t^{n}},\psi(n)\right)\cap C(b,D)\right)  
&  \ \subseteq \ \bigcup_{p\in G_{n} }\left(B\left(\frac{p}{t^{n}},\psi(n)\right)\cap C(b,D)\right) 
 \\[2ex]
&  \ \subseteq   \ \bigcup_{q=0}^{(b^{\alpha_{1}})^{n}}\left(B\left(\frac{q}{(b^{\alpha_{1}})^{n}},\psi(n)\right)\cap C(b,D)\right).
\end{aligned}
\end{equation*}
Therefore,
\begin{equation}
\begin{aligned}
\bigcup_{p=0}^{t^{n}}\left(B\left(\frac{p}{t^{n}},\psi(n)\right) \cap C(b,D)\right)& \ = \ \bigcup_{p\in\Gamma_{n}(\psi)}\left(B\left(\frac{p}{t^{n}}, \psi(n) \right)\cap C(b,D)\right) \nonumber \\[2ex]
& \ \subseteq \ \bigcup_{q=0}^{(b^{\alpha_{1}})^{n}}\left(B\left(\frac{q}{(b^{\alpha_{1}})^{n}},\psi(n)\right)\cap C(b,D)\right)  \label{lkj}
\end{aligned}
\end{equation}

On the other hand, in view of  \eqref{equality} it follows  that  $b^{\alpha_{1}}\mid t$, and  so
$$\bigcup_{q=0}^{(b^{\alpha_{1}})^{n}}\left(B\left(\frac{q}{(b^{\alpha_{1}})^{n}},\psi(n)\right)\cap C(b,D)\right)\subseteq\bigcup_{p=0}^{t^{n}}\left(B\left(\frac{p}{t^{n}},\psi(n)\right)\cap C(b,D)\right).$$
Thus, together with \eqref{lkj} we obtain equality;  that is  
$$\bigcup_{q=0}^{(b^{\alpha_{1}})^{n}}\left(B\left(\frac{q}{(b^{\alpha_{1}})^{n}},\psi(n)\right)\cap C(b,D)\right) =\bigcup_{p=0}^{t^{n}}\left(B\left(\frac{p}{t^{n}},\psi(n)\right)\cap C(b,D)\right).$$
The upshot is that
\begin{equation*}
\begin{aligned}
W_{t}(\psi)\cap C(b,D)& \ = \ \bigcap_{N=1}^{\infty}\bigcup_{n=N}^{\infty}\bigcup_{p=0}^{t^{n}}\left(B\left(\frac{p}{t^{n}},\psi(n)\right)\cap C(b,D)\right) \\[2ex] &   \ =  \ \bigcap_{N=1}^{\infty}\bigcup_{n=N}^{\infty}\bigcup_{q=0}^{(b^{\alpha_{1}})^{n}}\left(B\left(\frac{q}{(b^{\alpha_{1}})^{n}}, \psi(n)\right)\cap C(b,D)\right) \\[2ex] & \ = \ W_{b^{\alpha_{1}}}(\psi)\cap C(b,D) \, ,
\end{aligned}
\end{equation*}
as desired.  This completes the proof of part (ii). 
\textrm{}
\end{proof}

\bigskip

\bigskip

\begin{proof}[\rm \textbf{Proof of Theorem $\ref{11}$}] 
We prove Theorem $\ref{11}$ by considering two separate cases.

\medskip 

\noindent $\bullet$ \textbf{Case 1}: \textit{$\psi(n)\leq\frac{1}{2}t^{-n}$ for all $n$ sufficiently large. }  In the case, the proof  is similar to the proof of part (ii) of Theorem~\ref{12sv}. The only difference if that instead of using part (ii) of Proposition~\ref{ub}  we use part (i).  For completeness we give the details. For any $\delta>0$, there exists $N_{0}\in\mathbb{N}$, such that for all  $N\geq N_{0}$,
$$ \left\{
 B\left(\frac{p}{t^{n}},\psi(n)\right)\cap A: \ n\geq N,p\in\Gamma_{n}(\psi,A)\right\}$$
is a $\delta$-cover of $W_{t}(\psi)\cap A$. It thus follows on exploiting Proposition $\ref{ub}$ (i), that  for any $s>0$ and any $N\geq N_{0}$,
\begin{equation*}
\mathcal{H}_{\delta}^{s}(W_{t}(\psi)\cap A)\leq\sum_{n=N}^{\infty}\sum_{p\in\Gamma_{n}(A,\psi)}(2\psi(n))^{s}=\sum_{n=N}^{\infty}(2\psi(n)) ^{s}\#\Gamma_{n}(A,\psi)\ll\sum_{n=N}^{\infty}\psi(n)^{s}N_{t^{-n}}(A).
\end{equation*}
The upshot is  that if  $\sum_{n=1}^{\infty}\psi(n)^{s}N_{t^{-n}}(A) $ converges,  then the sum on the right hand side tends to zero as $N \to \infty$, and so by definition
$\mathcal{H}^{s}(W_{t}(\psi)\cap A =0    $ as desired. 

\bigskip

\noindent $\bullet$ \textbf{Case 2}: \textit{$\psi(n)>\frac{1}{2}t^{-n}$ for infinitely many $n\in\mathbb{N}$.}
 In this case, it follows from  the definition of $W_{t}(\psi)$ that $W_{t}(\psi)=[0,1]$ and so  $$W_{t}(\psi)\cap A=A.$$ Therefore, it suffices to show that 
$$\mathcal{H}^{s}(A)=0  \quad   {\rm if }  \quad  \sum\limits_{n=1}^{\infty}\psi(n)^{s}N_{t^{-n}}(A)  < \infty \, .   $$
 With this in mind, consider the modified function $\psi^*$ defined  by  $\psi^{*}(n)=\min(\psi(n),\frac{1}{2}t^{-n})$.  Then,  by construction $\psi^{*}(n)\leq\frac{1}{2}t^{-n}$ for all $n\in\mathbb{N}$  and so 
$$\sum\limits_{n=1}^{\infty}\psi^{*}(n)^{s}N_{t^{-n}}(A) \ \leq \ \sum\limits_{n=1}^{\infty}\psi(n)^{s}N_{t^{-n}}(A)<\infty.$$
Hence, by the conclusion of \textbf{Case 1}, we obtain that
$$\mathcal{H}^{s}(W_{t}(\psi^{*})\cap A)=0.$$
Furthermore, since $\psi(n)>\frac{1}{2}t^{-n}$ for infinitely many $n\in\mathbb{N}$, it follows that
$$\psi^{*}(n)=\frac{1}{2}t^{-n}$$
for infinitely many $n\in\mathbb{N}$ and so  
\begin{equation}\label{35sv}
 [0,1]\setminus\left\{\frac{p+\frac{1}{2}}{t^{n}}:n\in\mathbb{N} \, , \ 0\leq p\leq t^{n}-1\right\} \ \subseteq  \ W_{t}(\psi^{*}).   
\end{equation}
Here we use the fact that for any $x$ in the left  hand side set we have $\|t^{n}x\|< 1/2 
$ for all $n\in\mathbb{N}$. The upshot of $\eqref{35sv}$ is that $W_{t}(\psi^{*}) $ coincides with $[0,1]$ up to a countable set.  Consequently,    
$$\mathcal{H}^{s}(A)=\mathcal{H}^{s}(W_{t}(\psi^{*})\cap A)=0.$$
This completes the proof.
\textrm{}
\end{proof}

\noindent \textbf{Acknowledgments.} The authors thank Junjie Huang, Ruofan Li and Yufeng Wu for helpful discussions and suggestions. BL was supported by National Key R\&D Program of China (No. 2024YFA1013700), NSFC 12271176 and Guangdong Natural Science Foundation 2024A1515010946.
SV would like to warmly thank Bing Li for the kind invitation to Guangzhou as part of the South China University of Technology (SCUT) Overseas Lecturer Programme. This paper grew out of discussions during the visit.

\end{document}